\documentclass[a4paper]{amsart}
\usepackage{graphicx}
\usepackage[colorlinks, linkcolor= blue, citecolor= red]{hyperref}
\usepackage{pdfsync}
\usepackage{color}

\usepackage{lmodern}

\usepackage[T1]{fontenc}
\usepackage[utf8]{inputenc}

\usepackage[all]{xy}

\usepackage{amsmath, amssymb, amsfonts, amscd, amsthm, mathrsfs}

\renewcommand{\phi}{\varphi}
\newcommand{\C}{\mathbb{C}}
\newcommand{\z}{\mathbb{Z}}

\newcommand{\q}{\mathbb{Q}}
\newcommand{\f}{\mathbb{F}}

\newcommand{\p}{\mathbb{P}^1}

\newcommand{\sets}{{\mathfrak{Sets}_{xy}}}

\newcommand{\etq}{\mathfrak{Etale}(\bar\q (x))}
\newcommand{\cat}{\mathfrak{Dessins}}
\newcommand{\ct}{\mathfrak{C}}

\newcommand{\qb}{ {\overline{\q}} }
\newcommand{\gal}{\operatorname{Gal}(\qb / \q)}

\newcommand{\gt}{\mathcal{GT}}
\newcommand{\gtz}{\gt_{\!\!1} }
\newcommand{\rGT}{\widehat{\operatorname{GT}}}

\newcommand{\bs}{\backslash}
\newcommand{\gb}{\overline{G}}
\newcommand{\pairs}{\mathscr{P}}
\newcommand{\pc}{\pairs_c}

\renewcommand{\bar}{\overline}
\newcommand{\A}{\mathscr{A}}
\newcommand{\s}{\mathscr{S}}

\newcommand{\gap}[1]{{\tt #1}}

\newtheoremstyle{pedro}{}{}{\itshape}{}{\sc}{~--}{ }{\thmname{#1}\thmnumber{ #2}\thmnote{ (#3)}}

\newtheoremstyle{pedrodef}{}{}{}{}{\sc}{~--}{ }{\thmname{#1}\thmnumber{ #2}\thmnote{ (#3)}}

\theoremstyle{pedro}
\newtheorem{lem}{Lemma}[section]

\newtheorem{thm}[lem]{Theorem}

\newtheorem{prop}[lem]{Proposition}

\newtheorem{coro}[lem]{Corollary}

\theoremstyle{remark}

\newtheorem{rmk}[lem]{Remark}

\theoremstyle{pedrodef}

\title[Computations with~$\gt$]{The
  Grothendieck-Teichmüller group \\of a finite group and $G$-dessins d'enfants}

\author{Pierre Guillot}

\address{
Universit\'{e} de Strasbourg \& CNRS\\
Institut de Recherche Math\'{e}matique Avanc\'{e}e\\
7~Rue Ren\'{e} Descartes\\
67084 Strasbourg, France}

\email{guillot@math.unistra.fr}

\let\oldtocsection=\tocsection
\let\oldtocsubsection=\tocsubsection
\let\oldtocsubsubsection=\tocsubsubsection

\renewcommand{\tocsection}[2]{\hspace{0em}\oldtocsection{#1}{#2}}
\renewcommand{\tocsubsection}[2]{\hspace{2em}\oldtocsubsection{#1}{#2}}
\renewcommand{\tocsubsubsection}[2]{\hspace{2em}\oldtocsubsubsection{#1}{#2}}

\numberwithin{equation}{section}

\begin{document}

\maketitle

\begin{abstract}
For each finite group~$G$, we define the {\em Grothendieck-Teichmüller group} of~$G$, denoted~$\gt(G)$, and explore its properties. The theory of dessins d'enfants shows that the inverse limit of~$\gt(G)$ as~$G$ varies can be identified with a group defined by Drinfeld and containing~$\gal$. 

We give in particular an identification of~$\gt(G)$, in the case when~$G$ is simple and non-abelian, with a certain very explicit group of permutations that can be analyzed easily. With the help of a computer, we obtain precise information for~$G= PSL_2(\f_q)$ when $q \in \{ 4,$ $ 7,$ $ 8,$ $ 9,$ $ 11,$ $ 13,$ $ 16,$ $ 17,$ $ 19 \}$, and we treat~$A_7$, $PSL_3(\f_3)$ and~$M_{11}$. 

In the rest of the paper we give a conceptual explanation for the technique which we use in our calculations. It turns out that the classical action of the Grothendieck-Teichmüller group on dessins d'enfants can be refined to an action on ``$G$-dessins'', which we define, and this elucidates much of the first part.

\bigskip
\noindent{\bfseries Status:} this version should be nearly identical to that which is to appear in the proceedings volume {\em Symmetry in Graphs, Maps and Polytopes} (SIGMAP), Springer. (Precise reference not available just now.)
\end{abstract}

\section{Introduction}

Suppose that~$\Gamma $ is a finite group, generated by two distinguished elements~$x$ and~$y$, and such that \begin{enumerate}
\item[(i)] $\Gamma $ has an automorphism~$\theta $ such that~$\theta (x) = y$ and~$\theta (y) = x$,
\item[(ii)] $\Gamma$ has an automorphism~$\delta $ such that~$\delta (x)=z$ and~$\delta (y) = y$, where~$z$ is the element such that~$xyz=1$.
\end{enumerate}

In this situation we define a subgroup~$A(\Gamma )\subset Aut(\Gamma )$ as follows : an element~$\phi \in Aut(\Gamma )$ belongs to~$A(\Gamma )$, by definition, when
\begin{enumerate}
\item $\phi(x)$ is a conjugate of~$x^k$ for some integer~$k$,
\item $\phi$ commutes with~$\theta $ and~$\delta $ in~$Out(\Gamma )$.
\end{enumerate}

(It follows that~$\phi(y)$ is a conjugate of~$y^k$, and likewise for~$z$.) The image of~$A(\Gamma )$ in~$Out(\Gamma )$ will be denoted by~$\A(\Gamma )$.

For any finite group~$G$ at all, we shall see that there is a way to construct a group~$\gb$ satisfying (i) and (ii), so that it can play the role of~$\Gamma $ (and moreover~$\bar \gb = \gb$). Thus it makes sense to define~$\gt(G) := \A(\gb)$. We call it the {\em Grothendieck-Teichmüller group} of~$G$, and the present paper is dedicated to the study of its properties. We start with a few words of motivation and background.

How~$\gt(G)$ varies with~$G$ is a discussion which we postpone; for the time being, we take it for granted that it is possible to form the inverse limit
\[ \gt := \lim_G \gt(G) \, .   \]
In~\cite{pedro} we proved the central (for us) result that there is a monomorphism
\[ \gal \longrightarrow \gt \, .   \]
Thus~$\gt$, with its very brief definition,  gives a group-theoretic angle to the study of the absolute Galois group~$\gal$ of the field~$\q$. A very first step towards understanding~$\gt$ is to provide information on~$\gt(G)$ for some individual choices of~$G$, and this is what we propose to do here.

As an aside, the reader will probably find it useful to know that 
\[ \lim_G Out(\gb) \cong Out( \hat F_2) \, ,   \]
where~$\hat F_2$ is the profinite completion of the free group~$F_2$ on two generators. Thus~$\gt$ can be seen as a certain subgroup of~$Out(\hat F_2)$, and one can show that it can be lifted to a subgroup of~$Aut( \hat F_2)$. Also, let us indicate that~$\gt$ coincides with the group denoted~$\rGT _0$ by Drinfeld in~\cite{drinfeld} (we shall have nothing to say about the subgroup~$\rGT \subset \rGT_0$, also considered by Drinfeld). All this, and more, is proved in~\cite{pedro}. 

A good deal of the present paper will in fact pertain to~$\gtz(G)$, which is the subgroup of~$\gt(G)$ obtained by restricting condition (1) above to~$k=1$ only. One can show that there is a monomorphism 
\[ \gal ' \longrightarrow \gtz := \lim_G \gtz(G) \, ,   \]
where~$\gal'$ is the derived subgroup of~$\gal$. So~$\gtz$ can potentially give us information on~$\gal'$ just like~$\gt$ can give us information on~$\gal$, and of course the abelianization~$\gal / \gal' \cong \hat{\z}^\times$ is well-understood. Here~$\hat{\z}^\times$ is the group of units in the profinite completion of~$\z$. 

The following simple example should illuminate the situation. If~$G = C_n$, the cyclic group of order~$n$, we have $\bar{C_n} \cong C_n \times C_n$ with its canonical pair of generators. Then~$\gt(C_n) \cong \left( \z/n \right)^\times$, directly from the definition, while~$\gtz(C_n)$ is trivial. Letting~$n$ vary, we can take the inverse limit and obtain 
\[ \gal \longrightarrow \lim_n \gt(C_n) \cong \hat{\z}^\times \, .  \]
In turn, this homomorphism can be identified with the celebrated {\em cyclotomic character}, whose kernel is~$\gal'$. In a sense, consideration of cyclic groups accounts for what is abelian in~$\gal$, and we must turn to non-abelian groups and their~$\gt$ to proceed further.

\[ \star \star \star  \]

The Grothendieck-Teichmüller group is strongly related to the theory of {\em dessins d'enfants}, which are the object of many papers in these Proceedings (some information on dessins is given below in this Introduction, and more is said in~\S\ref{sec-dessins}). On the one hand one uses dessins in order to construct the homomorphism from~$\gal$ into~$\gt$ and show that it is injective. On the other hand, the group~$\gt$ can be used to shed light on the action of~$\gal$ on (isomorphism classes) of dessins. Indeed, when trying to predict whether two dessins belong to the same Galois orbit, one starts by checking a few combinatorial properties which they must have in common: the same number of black vertices, the same number of white vertices, the same number of faces, and the same ``monodromy group'', for example.

All these are subsumed by the following statement: the action of~$\gal$ on those dessins with monodromy group~$G$ factors via 
\[ \gal \longrightarrow \gt \longrightarrow \gt(G) \, .   \]
For simplicity, say that one is interested in {\em regular} dessins, those with ``maximal symmetry''. Then the regular dessins with monodromy group (or automorphism group) $G$ are in bijection with those normal subgroups~$N$ of~$\gb$ such that~$\gb/N \cong G$, and the action of~$\gal$ factors through the natural action of~$\gt(G)$ on these. The combinatorial features above can be recovered from this, and more. This motivates the computation of~$\gt(G)$ for a single group~$G$ individually.

An example of a finer statement which one can make about the Galois action is the following: if~$\gtz(G) = 1$, then~$\gal'$ acts trivially on the set of dessins with monodromy~$G$. In different terms, the {\em moduli field} of such a dessin, that is, the number field~$F$ whose fixed subgroup in~$\gal$ is the stabilizer of the dessin, is an {\em abelian} extension of~$\q$. (The field~$F$ is strongly related to, though sometimes smaller than, the number fields over which the dessin can be defined.) 

During the SIGMAP conference, Gareth Jones asked for examples of regular dessins with non-abelian moduli field. A hint for those trying to answer the question is thus that the monodromy group~$G$ must satisfy~$\gtz(G) \ne 1$. In the course of this paper we shall see that this rules out~$G= A_5$ and~$G = D_n$ when~$n$ is divisible by~$4$, among others.

\[ \star \star \star  \]

Let us now describe the contents of the paper. It is in Section~\ref{sec-generalities} where, after expanding on the definitions above, we prove that properties of~$G$ are reflected in properties of~$\gtz$ (but not~$\gt$). For example we establish:

\begin{thm}
If~$G$ is a~$p$-group for some prime~$p$, then so is~$\gtz(G)$; if~$G$ is nilpotent, then so is~$\gtz(G)$. 

The group~$\gt(G) / \gtz(G)$ is abelian, with exponent dividing that of~$\left(\z/N\right)^\times$, where~$N$ is the order of~$x$ or~$y$ in~$\gb$ (in particular, this exponent may not be a power of~$p$ when~$G$ is a~$p$-group).
\end{thm}

In Section~\ref{sec-simple} we define a new group~$\s(G)$. We hasten to add that when~$G$ is non-abelian and simple we shall prove that there is an isomorphism~$\gtz(G) \cong \s(G)$, so the material in that section can be seen at least as a study of the ``simple case''. However~$\s(G)$ is defined for all~$G$, and it is a much easier group to deal with than~$\gtz(G)$. It is described as the intersection, in a permutation group, of a Young subgroup and the centralizer of a few explicit permutations. The first virtue of~$\s(G)$ is that it is easy to reason with, leading for example to the next result:

\begin{thm}
Let~$G$ be a finite, simple, non-abelian group, and let~$m$ be the size of the largest conjugacy class in~$G$. A simple factor occuring in~$\gtz(G)$ must be isomorphic to either: \begin{itemize}
\item $C_2$,
\item $C_3$,
\item a subquotient of~$Out(G)$,
\item an alternating group~$A_s$ where~$s \le \frac{ m^2} {|G|}$.
\end{itemize}

\end{thm}

(We stress that the theorem mentions~$Out(G)$, not~$Out(\gb)$ which is much bigger and would make for a tautological statement.)

It is also easy to compute explicitly with~$\s(G)$. The reader should keep in mind that a computer, unleashed after~$\gtz(G)$ by a direct, brute force approach, will not be able to finish its task within a day or without exceeding the memory on a group~$G$ whose order is much bigger than~$32$. Relying only on naive calculations, the author has yet to see a completed example for which the order of~$\gtz(G)$ is anything but~$1, 2, 3, 4, 5, 6, 7$. By contrast, the machinery of~$\s(G)$ has allowed us to treat, for example, the case of the Mathieu group~$M_{11}$ of order~$7920$, yielding:

\begin{thm}
The direct product of the simple factors of~$\gtz(M_{11})$ is 
\begin{multline*}
 C_2^{465} \times C_3 ^{46} \times A_5^{10} \times A_6^9 \times A_7^{10} \times A_8^4 \times A_9^4 \times A_{10}^5 \times A_{11}^5 \times A_{12} \times A_{14}^2 \times A_{15}^4 \times A_{16} \times A_{17}^3 \times A_{18}^{12}\\ \times A_{19} \times A_{20}^2 \times A_{23} \times A_{28} \times A_{31} \times A_{33}^2 \, . 
\end{multline*}
Accordingly, the order of~$\gtz(M_{11})$ is $2^{1141} \cdot 3^{407} \cdot 5^{165} \cdot 7^{98} \cdot 11^{43} \cdot 13^{34} \cdot 17^{23} \cdot 19^{8} \cdot 23^{5} \cdot 29^{3} \cdot 31^{3}$.
\end{thm}

We also give a complete description of~$\gtz(PSL_2(\f_q))$ for $q \in \{ 4,$ $ 7,$ $ 8,$ $ 9,$ $ 11,$ $ 13,$ $ 16,$ $ 17,$ $ 19 \}$, and we treat~$A_7$ and~$PSL_3(\f_3)$. In Section~\ref{sec-gap} we explain some of the practicalities of the implementation with the open-source computer algebra system GAP.

To move on with our outline, let~$\pairs$ be the set of pairs~$(g, h) \in G$ such that~$\langle g, h \rangle = G$. We will see that there is a very natural action of~$\gt(G)$ on the set~$\pairs/ Aut(G)$. However, the development of the isomorphism between~$\gtz(G)$ and~$\s(G)$ relies on the existence of an action of~$\gt(G)$ on~$\pc$, the set of orbits in~$\pairs$ under the action of the inner automorphisms only (the letter~$c$ is for ``conjugation''). At first sight this appears rather mysterious, and the arguments are {\em ad hoc}. In section~\ref{sec-dessins} we give a conceptual explanation.  

The key is to bring dessins d'enfants into the picture. Here we must recall that a dessin is essentially a bipartite graph drawn on a compact, oriented surface in such a way that the complement of the graph is a union of topological discs. The (isomorphism classes of) dessins d'enfants are in bijection with many other sets of (isomorphism classes of) objects, notably algebraic curves over~$\qb$ with a certain ramification property, or étale algebras over~$\qb(x)$, again with a ramification property. The group~$\gal$ acts naturally on étale algebras, and this is turned into an action on dessins {\em via} the said bijection.

In~\cite{pedro} (which is our reference for dessins), we prove that dessins form a category~$\cat$, and that the aforementioned bijections can be refined into equivalences of categories. Such a refinement may not seem to bring much new information at first sight, but it is not so. Indeed, with this formalism it is completely straightforward to define the category~$G \cat$ of {\em $G$-dessins}, that is, dessins equipped with an action of a fixed group~$G $; and we prove the following:

\begin{thm}
The group~$\gal$ acts on the set of isomorphism classes of objects in~$G \cat$, for any group~$G$.

Moreover, suppose we consider the regular~$G$-dessins~$X$ in~$G \cat$ such that the action gives an isomorphism~$G \longrightarrow Aut(X)$. Then the set of isomorphism classes of such objects is naturally in bijection with~$\pc$, and the latter is endowed with an action of~$\gal$.
\end{thm}

(The word {\em regular} will be explained in the text.) It is now much more believable that~$\gt$ should act on~$\pc$; given that the action of~$\gt$ on dessins, when restricted to those dessins~$X$ such that~$Aut(X) \cong G$, factors {\em via}~$\gt(G)$, we should not be overly surprised by the discovery made in Section~\ref{sec-simple} that~$\gt(G)$ does act on~$\pc$.

\section{Generalities} \label{sec-generalities}

We start by expanding on the definitions given in the Introduction. We define~$\gb$, the group~$\gt(G)$ as a subgroup of~$Out(\gb)$, explain the relationship with~$\gal$, and prove the most basic properties.

\subsection{The group~$\gb$} \label{subsec-generalities-bato}

Let~$G$ be a finite group. Whenever~$N$ is a subgroup of a group~$\Gamma $, it will be convenient to say that~$N$ {\em has index~$G$ in~$\Gamma $} when (i) $N$ is normal in~$\Gamma $ and (ii) there is an isomorphism~$\Gamma / N \cong G$. 

Writing~$F_2 = \langle x, y \rangle$ for the free group on two generators~$x$ and~$y$, we call~$N_G$ the intersection of all the subgroups of~$F_2$ having index~$G$. There are finitely many of these, so the group~$\gb := F_2 / N_G$ is finite. We usually write~$x$ and~$y$ for the images of the generators of~$F_2$ in~$\gb$, since no confusion should arise. 

The following lemma is almost trivial.

\begin{lem} \label{lem-basic-props-gb}
$\gb$ has the following properties: \begin{enumerate}
\item The intersection of all the subgroups of~$\gb$ having index~$G$ is trivial.
\item If~$\Gamma $ is any group such that the intersection of all its subgroups of index~$G$ is trivial, and if~$x'$ and~$y'$ are generators of~$\Gamma $, then there is a homomorphism~$\gb \to \Gamma $ mapping~$x$ to~$x'$ and~$y$ to~$y'$.
\item If~$x'$ and~$y'$ are generators for~$\gb$, then there is an automorphism of~$\gb$ mapping~$x$ to~$x'$ and~$y$ to~$y'$.
\end{enumerate}
\end{lem}

We turn to the description of a concrete ``model'' for~$\gb$. The key observation is that subgroups of~$F_2$ of index~$G$ are in bijection with the orbits of~$Aut(G)$ on the set~$\pairs$ of pairs of generators for~$G$; the bijection sends a pair~$(x', y')$ to the kernel of the map~$F_2 \to G$ sending~$x$ to~$x'$ and~$y$ to~$y'$. 

Based on this, we select pairs~$(x_1, y_1), \ldots, (x_r,y_r)$ forming a system of representatives for the orbits of~$Aut(G)$, that is, with just one pair out of each orbit. (The number~$r = r(G)$ was much studied in~\cite{gareth}.) Consider then the subgroup~$\tilde G$ of~$G^r$ generated by~$x= (x_1, x_2, \ldots , x_r)$ and~$y = (y_1, y_2, \ldots, y_r)$. Then it is straightforward to show that~$\tilde G$ satisfies (2) of lemma~\ref{lem-basic-props-gb} (since the group~$\Gamma $ mentioned there embeds into~$G^r$). This property clearly characterizes~$\gb$ as a group with distinguished generators, so there must be an isomorphism~$\gb \cong \tilde G$ identifying the two elements which we have both called~$x$, and likewise for~$y$. For most of this paper we will consider~$\gb$ to be the subgroup of~$G^r$ just defined.

Let~$p_i$ be the projection onto the~$i$-th factor of~$G^r$, restricted to~$\gb$. It sends~$x$ to~$x_i$ and~$y$ to~$y_i$, so it is surjective and its kernel~$K_i$ has index~$G$. The various~$K_i$'s are distinct (by choice of the pairs~$(x_i, y_i)$), so they must constitute the~$r$ different subgroups of index~$G$ in~$\gb$. In particular they form a characteristic family of subgroups, that is, for any~$\phi \in Aut(\gb)$ we must have~$\phi(K_i) = K_{\sigma (i)}$ for some permutation~$\sigma \in S_r$. 

Finally we note that~$\bar \gb = \gb$. Indeed, if we try to construct the model for~$\bar \gb$ as we have just done with~$\gb$, then property (3) of lemma~\ref{lem-basic-props-gb} leaves us only one pair to consider; in other words, $r(\gb) = 1$.

\subsection{The group~$\gt(G)$}

By (3) of lemma~\ref{lem-basic-props-gb}, the group~$\gb$ has an automorphism~$\theta $ with~$\theta (x) = y$ and~$\theta (y) = x$; likewise, $\gb$ possesses an automorphism~$\delta $ with~$\delta (x) = y^{-1} x^{-1}$ and~$\delta (y) = y$. 

Consider now the elements~$\phi \in Aut(\gb )$ satisfying
\begin{enumerate}
\item $\phi(x)$ is a conjugate of~$x^k$ for some~$k$ prime to the order of~$\gb $,
\item $\phi$ commutes with~$\theta $ and~$\delta $ in~$Out(\gb )$.
\end{enumerate}
(It follows that~$\phi(y)$ is a conjugate of~$y^k$, and likewise~$xy$ is a conjugate of~$(xy)^k$.) These form a subgroup of~$Aut(\gb)$, and its image in~$Out(\gb)$ will be called~$\gt(G)$. 

Likewise, we can consider those automorphisms satisfying (1) for~$k=1$ only, as well as (2); they induce a normal subgroup~$\gtz(G)$ of~$\gt(G)$.

A complication to keep in mind is that there is no well-defined map on~$\gt(G)$ that would associate to~$\phi$ the number~$k$ as above: the latter is not unique, and not even unique modulo the order of~$x$, for some powers of~$x$ may well be conjugated to one another. In other words an element of~$\gtz(G)$ may have the property that~$\phi(x)$ is a conjugate of~$x^k$ for many values of~$k \ne 1$.

It is however true that when~$\phi(x)\sim x^k$ and~$\psi(x) \sim x^\ell$, then~$\psi \circ \phi (x) \sim x^{k \ell}$, where we write~$a \sim b$ when~$a$ and~$b$ are conjugate. In particular since~$\phi^{-1}$ is a power of~$\phi$, we note that~$\phi^{-1}(x) \sim x^{k'}$ where~$x^{kk'} \sim x$. If~$\psi^{-1}(x) \sim x^{\ell'}$ with~$x^{\ell \ell'} \sim x$, then the commutator~$[\phi, \psi]$ takes~$x$ to a conjugate of 
\[ x^{kk'\ell \ell'} = (x^{kk'})^{\ell \ell'} \sim x^{\ell \ell'} \sim x \, .   \]

We have proved that all commutators in~$\gt(G)$ must belong to~$\gtz(G)$. Thus we may state:

\begin{lem} \label{lem-gtz-abelian}
The group~$\gt(G) / \gtz(G)$ is an abelian group, of exponent dividing that of $\left(\z / N \z\right)^\times$, where~$N$ is the order of~$x$ (or~$y$) in~$\gb$. 
\end{lem}

The statement about the exponent follows from the fact that~$\phi(x) \sim x^k$ for some~$k \in \left(\z / N \z\right)^\times$, whenever~$\phi \in \gt(G)$. So~$\phi^n(x) \sim x^{k^n} = x$ when~$k^n = 1$ mod $N$, and then~$\phi^n \in \gt_1(G)$.

\subsection{Inverse limits} \label{subsec-inverse-limits}

If~$N$ is a normal subgroup of~$F_2$ of finite index, we can always find a~$G$ such that~$N_G \subset N$: indeed it suffices to take~$G = F_2 / N$. From this one can show that 
\[ \lim F_2 / N_G \cong \hat F_2 \, ,   \]
where~$\hat F_2$ is the profinite completion of~$F_2$. Here the inverse limit is over the directed set of all the subgroups of the form~$N_G$ (with their inclusions). Details for this, and everything else in the next few paragraphs, are provided in~\cite{pedro}.

When~$N_G \subset N_H$, we have a map~$\gb \to \bar H$, whose kernel is the intersection of all the subgroups of~$\gb$ having index~$H$. In particular, this kernel is a characteristic subgroup, and as a result we have an induced map 
\[ \gt(G) \longrightarrow \gt(H) \, .   \]
Thus it makes sense to talk about the inverse limit~$\lim \gt(G)$. Again the indexing set for the limit is the set of the various subgroups~$N_G$, but we prefer to write more suggestively 
\[ \lim_G \, \gt(G)  \]
which we call~$\gt$. We also put 
\[ \gtz := \lim_G \gtz(G) \, .   \]

\subsection{The Galois group of~$\q$}

In~\cite{pedro} we prove the existence of a monomorphism 
\[ \Phi \colon \gal \longrightarrow \gt  \]
which is the motivation for the study of~$\gt$. Moreover, if~$\lambda \in \gal$ and if~$\phi = \Phi(\lambda )$, then we can compute for any~$G$ an integer~$k$ such that~$\phi(x)$ and~$x^k$ are conjugate in~$\gb$: namely, let~$N$ be the order of~$x$, let~$\zeta = e^{\frac{2i\pi} {N}}$, and pick~$k$ such that~$\lambda (\zeta) = \zeta^k$.

We write~$\gal'$ for the derived subgroup of~$\gal$ (the closed subgroup generated by the commutators). A celebrated result in number theory asserts that~$\gal'$ is precisely the subgroup of elements acting trivially on all the roots of unity (this is essentially the Kronecker-Weber theorem, see \cite{neukirch}, chapter 5, theorem 1.10) As a result, or simply as an application of lemma~\ref{lem-gtz-abelian}, there is also a monomorphism 
\[ \gal' \longrightarrow \gtz \, .   \]

It is surprising that the lemma below seems hard to prove without appealing to~$\gal$. It is never used on the sequel.

\begin{lem}
Let~$N$ be the order of~$x$ (or~$y$) in~$\gb$. Then for any integer~$k$ prime to~$N$, there is~$\phi \in \gt(G)$ such that~$\phi(x)$ is a conjugate of~$x^k$.
\end{lem}

\begin{proof}
Simply take~$\phi= \Phi( \lambda )$ where~$\lambda \in \gal$ has the appropriate effect on roots of unity.
\end{proof}

\subsection{$p$-groups and nilpotent groups}

\begin{prop}
If~$G$ is a~$p$-group, then so is~$\gtz(G)$.
\end{prop}

\begin{proof}
First note that~$\gb$ is itself a~$p$-group, being a subgroup of~$G^r$. Let~$A'(\gb)$ denote the preimage of~$\gtz(\gb)$ in~$Aut(\gb)$. If~$A'(\gb)$ is not a~$p$-group, then it contains an element~$\phi$ whose order is a prime~$\ell \ne p$.

Consider the elementary abelian~$p$-group~$E =\gb / \Phi(\gb)$, where~$\Phi(\gb)$ is the Frattini subgroup of~$\gb$, generated by the images~$\bar x$ and~$\bar y$ of~$x$ and~$y$. The induced action of~$\phi$ on~$E$ is then trivial. It follows from~\cite{isaacs}, Corollary 3.29 (a result sometimes referred to as the Burnside basis theorem), that the action of~$\phi$ on~$\gb$ is trivial, violating the assumption that the order of~$\phi$ is~$\ell$. This contradiction shows that~$A'(\gb)$ is a~$p$-group, and so also is~$\gtz(G)$.
\end{proof}

\begin{prop}
If~$G$ and~$H$ have coprime orders, we have~$\bar{G\times H} \cong \bar G \times \bar H$ and~$\gt(G\times H) \cong \gt(G) \times \gt(H)$, as well as~$\gtz(G\times H) \cong \gtz(G) \times \gtz(H)$.
\end{prop}

\begin{proof}
We start with a remark. Whenever a group~$N$ has index~$G \times H$ in a group~$\Gamma $, then we can write~$N = N' \cap N''$ where~$N'$ has index~$G$ and~$N''$ has index~$H$, clearly. Now suppose the orders of~$G$ and~$H$ are coprime, and let us prove the converse. If~$N = N' \cap N''$ for such~$N'$ and~$N''$, then~$\Gamma /N$ injects in~$\Gamma / N' \times \Gamma /N'' \cong G \times H$, and its image surjects onto both~$G$ and~$H$. Thus the order of~$\Gamma /N$ is divisible by both~$|G|$ and~$|H|$ and so by their product, so that~$\Gamma/N \cong G \times H$, as we wished to show.

Applying this remark to the subgroups of the free group~$F_2$, we deduce that 
\[ N_{G \times H} = N_G \cap N_H \tag{*} \, .   \]
(Recall that~$N_G$ is the intersection of the subgroups of index~$G$, and likewise for~$N_H$ and~$N_{G \times H}$.) 

What is more, $\gb$ and~$\bar H$ also have coprime orders since they are subgroups of~$G^r$ and~$H^s$ respectively. Thus we may apply the remark again, and deduce from (*) that~$N_{G \times H}$ has index~$\gb \times \bar H$ (being the intersection of a group of index~$\gb$ and a group of index~$\bar H$). This shows that there is an isomorphism~$\bar{G \times H} \to \gb \times \bar H$.


Next we note that an automorphism of~$\bar G \times \bar H$ must be of the form~$\alpha \times \beta $ where~$\alpha \in Aut(\bar G)$ and~$\beta \in Aut(\bar H)$. It follows easily that~$\gt(G\times H) \cong \gt(G) \times \gt(H)$.
%
%
\end{proof}

\begin{coro}
If~$G$ is nilpotent, then so is~$\gtz(G)$.
\end{coro}

\begin{proof}
A finite group is nilpotent precisely when it is a direct product of~$p$-groups.
\end{proof}

\section{An elementary example: dihedral groups}

In this section we present a computation of~$\gtz(D_n)$ where~$D_n$ is the dihedral group of order~$2n$ (with details only when~$n$ is odd). It is simple enough to be carried out ``by hand'' to the end, while by contrast the methods developped in the sequel ultimately rely on computers when put to practice. We believe that many features of~$\gtz(G)$ are already visible here.

So let~$s$ and~$t$ be involutions generating~$G = D_n$, and let~$R= st$, so that the~$2n$ elements of~$G$ are the ``rotations''~$R^m$ and the involutions~$sR^m$, for~$0 \le m < n$. It is easily seen that a pair of generators~$(x_1, x_2)$ for~$G$ can be taken by an automorphism to one of~$(s, t)$, $(R, t)$ or~$(s, R)$. As a result, $\gb$ is the subgroup of~$G^3$ generated by~$x = (s, R, s)$ and~$y = (t, t, R)$.

From now on, we assume that~$n$ is odd, and we proceed to prove that~$\gtz(G)$ has order~$2$. We shall state the corresponding results for other values of~$n$ below.

\subsubsection*{Observations.}
First we describe the group~$\gb$ a to some extent. To do so, we observe that the abelianization of~$G$ is~$C_2$, so~$G^3$ has projects onto~$C_2^3 = \{ (\pm 1, \pm 1, \pm 1) \}$, and looking at the images of~$x$ and~$y$ we see that~$\gb$ maps onto the subgroup of elements~$(a, b, c)$ with~$abc= 1$. Thus the index of~$\gb$ in~$G^3$ is at least~$2$. However, since~$x^2 = (1, R^2, 1)$ and the order of~$R$ is odd, we see that~$(1, R, 1) \in \gb$ ; likewise, starting with~$y$ and~$xy$, we see that~$(1, 1, R)$ and~$(R, 1, 1)$ are in~$\gb$. There is thus a subgroup~$A \cong C_n^3 \subset \gb$, and the order of~$\gb$ is a multiple of~$n^3$. Finally, note that~$A$ is normal in~$G^3$ and hence also in~$\gb$, and the quotient~$\gb / A$ is easily seen to be~$C_2 \times C_2$, so the order of~$\gb$ is~$4n^3$ and its index in~$G^3$ is just~$2$. In passing we have established a recipe for checking whether an element~$(\alpha ,  \beta ,  \gamma ) \in G^3$ belongs to~$\gb$: namely, this is the case if and only if there are an even number of involutions among~$\alpha, \beta , \gamma $.

It will be useful to know the centralizer~$C_{\gb}(y)$ of~$y$ in~$\gb$. First off, the centralizer in~$G^3$ is~$C_{G^3}(y) = C_2\times C_2 \times C_n$ generated by~$(t, 1, 1)$, $(1, t, 1)$ and~$(1, 1, R)$, so it has order~$4n$. Since the order of~$y$ is~$2n$ (using that~$n$ is odd), and since there are elements in~$C_{G^3}(y)$ which are not in~$\gb$, such as~$(t, 1, 1)$, we conclude that~$C_{\gb}(y)= \langle y \rangle$.


\subsubsection*{Choices for~$\phi$.} 

Now let~$\phi \in Aut(\gb)$ represent an element of~$\gtz(G)$. Composing with an inner automorphism if necessary, we may assume that~$\phi(y) = y$, and we know that~$\phi(x) = x'$ can be conjugated to~$x$ within~$\gb$, and so also within~$G^3$. Put~$x' = (s', R', s'')$, where~$s'$ and~$s''$ are involutions and~$R'= R^{\pm 1}$ is a rotation. 

Now suppose~$\psi$ is another such automorphism of~$\gb$, with~$\psi(y) = y$ and~$\psi(x) = x''$, a conjugate of~$x$. Then $\phi$ and~$\psi$ differ by an inner automorphism, or equivalently represent the same element in~$\gtz(G)$, if and only if~$x'$ can be conjugated to~$x''$ by an element of~$C_{\gb}(y) = \langle y \rangle$. 

Here we point out that all the involutions in~$G$ are conjugate, and indeed can be conjugated to one another using a power of~$R$: using the notation~$a^b$ for~$b^{-1} a b$, this follows from~$(sR^i)^R = sR^{i+2}$ and the fact that the order of~$R$ is odd. Given that~$y=(t, t, R)$, we can clearly conjugate~$x'$ by a power of~$y$ to obtain an element whose third coordinate is any involution we want, say~$s$. In other words, we may assume that~$s'' = s$ without loss of generality. Conjugating further by~$y^n = (t, t, 1)$ if necessary, we may assume that~$R' = R$, that is~$x' = (s', R, s)$. Different choices for~$s'$ can only lead to different elements of~$\gtz(G)$.

We must have~$s' = sR^m$ for an integer~$m$ (taken mod~$n$). The next step is to show that there are only two possibilities for~$m$.

\subsubsection*{The condition involving~$\delta  $.} This will be imposed by the condition stating that~$\phi$ and~$\delta $ must commute in~$Out(\gb)$, by definition of~$\gtz(G)$. Recall that~$\delta (y) = y = (t, t, R)$ and~$\delta (x) = y^{-1} x^{-1} = (ts, tR^{-1}, R^{-1}s) = (R^{-1}, s, t)$. Pick a power of~$R$, say~$R^p$, such that~$s^{R^p} = t$. As the element~$(t, 1, R^p)$ commutes with~$y$, and~$(R, s, s)^{(t, 1, R^p)} = (R^{-1}, s, t)$, we conclude that 
\[ \delta (a, b, c) = (b, a, c)^{(t, 1, R^p)} \, ,  \]
for any~$(a, b, c) \in \gb$. Indeed, both sides of this equation define homomorphisms~$\gb \to G^3$, and they agree on~$x$ and~$y$. The attentive reader will notice that finding a simple expression for~$\delta$, replacing the definition in terms of the generators~$x$ and~$y$, is a silent but major theme in all the rest of the paper, and the same applies to~$\theta $.

We are now able to compute~$\delta (\phi(x)) = \delta (x') = (R^{-1}, -, -)$ (what happens with the second and third coordinates turns out to be irrelevant for the sequel, and would be distracting to look at). On the other hand~$\phi(\delta (x)) = y^{-1}(x')^{-1} = (ts', -, -) = (R^{m-1}, -, -)$. And of course~$\delta (\phi(y)) = \phi(\delta (y)) = y$.

The condition on~$\phi$ thus states the existence of~$c \in \gb$ such that (i) $y^c = y$, that is~$c$ centralizes~$y$, and (ii) $(R^{m-1}, -, -)^c = (R^{-1}, -, -)$. By the observation above, (i) implies $c \in \langle y \rangle$. The element~$c$, in particular, is of the form~$(1, -, -)$ or~$(t, -, -)$. 

Each possibility implies a value for~$m$. Indeed if~$c =(1, -, -)$, condition (ii) gives~$R^{-1} = R^{m-1}$ so that~$m=0$.  The case~$c= (t, -, -)$ yields~$R = R^{m-1}$ so that~$m= 2$.

\subsubsection*{Existence} We know now that there can be at most two elements in~$\gtz(G)$: the identity and the class of a potential automorphism~$\phi$ such that~$\phi(y) = y$ and~$\phi(x) = x' = (sR^2, R, s)$. To show that such an automorphism actually exists, we may simply consider conjugation by the element~$(t, 1, 1) \in G^3$, which does not belong to~$\gb$. 

We are left with the task of checking that~$\phi$ really defines an element of~$\gtz(G)$, that is, it must be verified that~$\phi$ and~$\theta $ commute up to an inner automorphism. Recall that~$\theta (x) = y$ and~$\theta (y) = x$. Using that~$x' = x (xy)^2$, a straightforward computation shows that we must find an element which simultaneously conjugates~$y (yx)^2 = (tR^{-2}, t, R)$ to~$y=(t, t, R)$ and~$x= (s, R, s)$ to~$x' = (sR^2, R, s)$. For this one may take~$(R, 1, 1)$.

We have proved the first part of the following proposition:

\begin{prop} \label{prop-dihedral-summary}
If~$n$ is odd, then the group~$\gtz(D_n)$ has order~$2$.

If~$n= 2k$ and~$k$ is odd, then the group~$\gtz(D_n)$ also has order~$2$. If~$k$ is even, then the group~$\gtz(D_n)$ is trivial.
\end{prop}

The rest of the proposition is left as a lengthy exercise. Note that when~$n= 2k$, the group~$\gb$ has order~$4k^3$ and so has index~$16$ in~$G^3$.

Let us say a word about the image of~$\gtz$ in~$\gtz(D_n)$. Let us use the notation~$s_n$, $t_n$ and~$R_n$ for the elements in~$D_n$ written~$s$, $t$, $R$
 up to now. There is a homomorphism~$D_{nm} \to D_n$ sending~$s_{nm}$ to~$s_n$, $t_{nm}$ to~$t_n$, and~$R_{nm}$ to~$R_n$. Clearly the induced homomorphism~$D_{nm}^3 \to D_n^3$ maps~$\bar D_{nm}$ onto~$\bar D_n$. It is a general fact, already mentioned in~\S\ref{subsec-inverse-limits}, that in this situation there is a map~$\gtz(D_{nm}) \to \gtz(D_n)$. 

The projection map~$\gtz \to \gtz(D_n)$ thus factors through~$\gtz(D_{nm})$ for any~$m$, in particular through~$\gtz(D_{4n}) = 1$. As a result, the image of~$\gtz$ in~$\gtz(D_n)$ is trivial, for all~$n$. What amounts essentially to the same thing, the inverse limit~$\lim_n \gtz(D_n)$ makes sense here, but sadly, it is trivial.

\section{The case of simple groups} \label{sec-simple}

For any finite group~$G$, we define a permutation group~$\s(G)$. When~$G$ is simple and non-abelian, we proceed to show that there is an isomorphism~$\gtz(G) \cong \s(G)$. This is used to analyse the possible simple factors in~$\gtz(G)$ in this case.

\subsection{Notation}

Let~$G$ be a finite group (shortly to be assumed simple and non-abelian, but not at the moment). The following notation will be used throughout this section. Let us emphasize that we make some arbitrary {\em choices} at the same time. 

Let~$\pairs$ denote the set of pairs of elements~$(g, h)$ generating~$G$. The group~$Aut(G)$ acts on~$\pairs$, and the set~$\pairs / Aut(G)$ of orbits has cardinality~$r$. It will be useful to also work with~$\pc$, the set of orbits under the sole action of the inner automorphisms. We see that~$Out(G)$ acts freely on~$\pc$, and~$\pc / Out(G) = \pairs / Aut(G)$. Thus the set~$\pc$ has cardinality~$r |Out(G)|$. (Please note that the actions considered here are on the left. In this section the composition on~$Aut(G)$ is~$\alpha \beta  = \alpha \circ \beta $.)

For each~$1 \le i \le r$ we choose a representative~$(x_i, y_i) \in \pairs$ for the~$i$-th orbit in~$\pairs/Aut(G)$, in some ordering.

The~$Aut(G)$-orbit of~$(g, h) \in \pairs$ will be denoted~$[g, h]$, while its orbit under~$Inn(G)$ will be written~$[g, h]_c$ (the brackets will never denote commutators in this section). In this notation the action of~$ \alpha \in Out(G)$ on~$[g, h]_c$ is~$\alpha \cdot [g, h]_c = [\alpha (g), \alpha (h)]_c$. The elements of~$\pc$ are precisely enumerated as~$[\alpha (x_i), \alpha (y_i)]_c$ for~$\alpha \in Out(G)$ and~$1 \le i \le r$. The following is immediate.

\begin{lem} \label{lem-pc-as-product}
There is a bijection of sets 
\[ \pc \longrightarrow Out(G) \times \pairs/Aut(G) \, ,  \]
sending~$ [\alpha (x_i), \alpha (y_i)]_c$ to the pair~$(\alpha, [x_i, y_i])$. It is equivariant with respect to the~$Out(G)$ actions, where on the right hand side the group~$Out(G)$ acts trivially on~$\pairs/Aut(G)$ and by left multiplication on itself.
\end{lem}

Finally, each pair~$(x_i, y_i)$ determines a unique homomorphism~$p_i \colon \gb \to G$ sending~$x$ and~$y$ to~$x_i$ and~$y_i$ respectively (recall that~$x$ and~$y$ are the canonical generators of~$\gb$). The kernel of~$p_i$ will be written~$K_i$.

\begin{rmk}
In the literature on dessins d'enfants or related group-theoretical topics, one often works with triples~$(x, y, z)$ of elements generating a finite group~$G$ and satisfying~$xyz = 1$. Our~$\pairs$ can be identified with the set of such triples, clearly, and~$\pc$ can be thought of as the set of triples up to simultaneous (triple) conjugation. Likewise the rest of this section could be developed with this (hardly different) point of view.
\end{rmk}

\subsection{An action of~$Out(\gb)$ on~$\pc$}

First recall (from the discussion in~$\S\ref{subsec-generalities-bato}$) that the~$K_i$'s form a characteristic family of subgroups in~$\gb$; in other words, for any~$\phi \in Aut(\gb)$ and any~$i$ there is a~$\sigma (i)$ such that~$\phi(K_i) = K_{\sigma (i)}$. The permutation~$\sigma \in S_r$ thus obtained from~$\phi$ may occasionally be denoted~$\sigma (\phi)$. 

Next, the composition~$\gb \stackrel{\phi}{\longrightarrow} \gb \stackrel{p_{\sigma (i)}}{\longrightarrow} G$ factors through~$p_i$, thus resulting in an automorphism~$G \to G$ which we denote~$\phi_i$. There is the composition formula 
\[ (\psi \circ \phi)_i = \psi_{\sigma (i)} \circ \phi_i \quad\textnormal{where}\quad \sigma = \sigma (\phi)\, .   \]
Also observe that when~$\phi$ is inner, the permutation~$\sigma (\phi)$ is the identity, and each~$\phi_i$ is inner.

We can now define an action of~$Aut(\gb)$ on~$Out(G) \times \pairs/Aut(G)$ by setting 
\[ \phi \cdot (\alpha , [x_i, y_i]) = (\alpha \phi_i^{-1} , [x_{\sigma (i)}, y_{\sigma (i)}]) \, .   \]
(On the second factor this is the natural action on~$\pairs/Aut(G)$, which can be identified with the set of the~$K_i$'s.) In this expression we have written~$\phi_i^{-1}$ for the class of this automorphism in~$Out(G)$. It is clear that this is indeed an and that it factors through~$Out(\gb)$.

Crucially, we notice that the action just defined commutes with that of~$Out(G)$ (by left multiplication on itself and trivially on~$\pairs/Aut(G)$).

By lemma~\ref{lem-pc-as-product}, we also have an action of~$Out(\gb)$ on~$\pc$, which commutes with the natural action of~$Out(G)$. We have in particular 
%
\[ \phi \cdot [x_i, y_i]_c  =  [\phi_i^{-1}(x_{\sigma (i)}), \phi_i^{-1} ( y_{\sigma (i)} )]_c \, ,  \]
which we will use more often than the general expression
%
\[ \phi \cdot \left( [\alpha (x_i), \alpha (y_i)]_c \right)  =  [\alpha \phi_i^{-1} (x_{\sigma (i)} ), \alpha \phi_i^{-1} (y_{\sigma (i)} )]_c  \, .  \]
(Commutation with~$Out(G)$ means that the first formula implies the second anyway.)

This discussion is summarized in the next proposition.  

\begin{prop} \label{prop-map-to-centraliser}
There is a homomorphism 
\[ Out(\gb) \longrightarrow C_S( Out(G) ) \, ,   \]
where~$S= S(\pc)$ is the symmetric group of the set~$\pc$, and~$C_S(Out(G))$ is the centralizer of~$Out(G)$ for the natural action. The corresponding action of~$Out(\gb)$ on~$\pc$ satisfies in particular 
%
\[ \phi \cdot [x_i, y_i]_c  =  [\phi_i^{-1}(x_{\sigma (i)}), \phi_i^{-1} ( y_{\sigma (i)} )]_c \, ,  \]

\end{prop}

(Note that when~$\phi \in Out(\gb)$, or~$\phi \in Aut(\gb)$, we simply write~$\phi \cdot [g,h]_c$ for the action.)

\begin{rmk}
The specific choices we have made for the elements~$x_i$ and~$y_i$ actually matter here. The curious reader may prove the following. Using the material below on simple groups, one can at least establish that when~$G$ is simple, the permutation~$\sigma (\phi)$ and the automorphisms~$\phi_i$ are uniquely defined (once one has a numbering of the elements of~$\pairs/Aut(G)$), and so the action on~$Out(G) \times \pc$ can be defined without making choices. However even in this case, the bijection of lemma~\ref{lem-pc-as-product} depends on choices.
\end{rmk}

We need to identify the permutations of~$\pc$ induced by certain specific elements of~$Aut(\gb)$. We start with the automorphism~$\theta$ of~$\gb$ which exchanges~$x$ and~$y$. The next lemma is perhaps not surprising, but its proof requires some care. 

\begin{lem} \label{lem-theta-no-surprises}
For all~$g, h \in G$, we have 
\[ \theta \cdot [g, h]_c = [h, g]_c \, .   \]
\end{lem}

\begin{proof}
Consider the following commutative diagram: 
\[ \begin{CD}
\gb @>{\theta }>> \gb \\
@V{p_i}VV               @VV{p_{\sigma (i)}}V \\
G  @>{\theta_i }>> G
\end{CD}
  \]
Recall that~$\theta (x) = y$, $\theta (y)= x$, $p_i(x)=x_i$, $p_i(y)= y_i$, and likewise for~$p_{\sigma (i)}$. Thus we see that~$\theta_i(x_i) = y_{\sigma (i)}$ and~$\theta _i(y_i) = x_{\sigma (i)}$, which we may profitably rewrite as~$\theta_i^{-1}(y_{\sigma (i)}) = x_i$ and~$\theta_i^{-1}(x_{\sigma (i)}) = y_i$. 

Following the definitions, we see that 
\[ \theta \cdot [x_i, y_i]_c = [y_i, x_i]_c \, .   \]
Thus the proposed formula is true at least when~$[g, h]_c = [x_i, y_i]_c$ for some~$i$. However the map~$[g, h]_c \mapsto [h, g]_c$ commutes with the action of~$Out(G)$, as does~$[g, h]_c \mapsto \theta \cdot [g, h]_c$, so these two maps have to agree.
\end{proof}

In the exact same vein, we have 

\begin{lem} \label{lem-delta-no-surprises}
For all~$g, h \in G$, we have 
\[ \delta  \cdot [g, h]_c = [h^{-1} g^{-1}, h]_c \, .   \]
\end{lem}

We leave the proof to the reader (recall that~$\delta (x)= y^{-1} x^{-1}$ and~$\delta (y) = y$). 

The next (and last) lemma involves the action of~$G \times G$ on~$\pc$ by conjugation.

\begin{lem} \label{lem-gtz-preserves-GxG}
Let~$\phi \in Aut(\gb)$ be such that~$\phi(x)$ is conjugate to~$x$, and~$\phi(y)$ is conjugate to~$y$. Then the action of~$\phi$ on~$\pc$ preserves the orbits of~$G \times G$.
\end{lem}

\begin{rmk}[on our cavalier use of the word ``orbit'' here] \label{rmk-orbits} While~$G\times G$ acts on itself by conjugation, the action does not restrict to the subset~$\pairs$. As a result it does not make sense to speak of the orbits of~$G \times G$ on~$\pairs$, let alone~$\pc$. However it does make sense to ask whether two elements of~$\pairs$ lie in the same~$G \times G$-orbit (that is, orbit on~$G \times G$); it also makes sense to ask whether two elements of~$\pc$ are the images of two elements of~$\pairs$ in the same~$G \times G$-orbit: very explicitly~$[a_1, b_1]_c$ and~$[a_2, b_2]_c$ are thus related if~$a_1$ and~$a_2$ are conjugate and~$b_1$ and~$b_2$ are conjugate, a relation which is well-defined.

This is how the notion of an ``orbit'' should be interpreted in the lemma and in related statements that follow.

In \S\ref{subsec-example-psl} we give an example where the elements of~$\pc$ lying in the same~$G \times G$-''orbit'' are grouped together into blocks; one of these blocks is of size~$10$ while~$G$ has order~$168$, showing that the blocks are not actual orbits.
\end{rmk}

\begin{proof}[Proof of lemma~\ref{lem-gtz-preserves-GxG}]
The action of~$Out(G)$ on~$\pc$ preserves the~$G \times G$-orbits, clearly, so it suffices to show that~$\phi \cdot [x_i, y_i]_c$ is in the same~$G\times G$-orbit as~$[x_i, y_i]_c$ for each index~$i$.

By assumption~$\phi(x) = x^g$ so~$\phi_i(x_i) = p_{\sigma (i)}(\phi(x)) = x_{\sigma (i)}^{g'}$ where~$g' = p_{\sigma (i)}(g)$. It follows that~$\phi_i^{-1}(x_{\sigma (i)})$ is conjugate to~$x_i$. Likewise, $\phi_i(y_{\sigma (i)})$ is conjugate to~$y_i$, and in the end we have indeed shown that~$\phi_i^{-1} \cdot [x_{\sigma (i)}, y_{\sigma (i)}]_c$ is in the~$G \times G$-orbit of~$[x_i, y_i]_c$.
\end{proof}

\subsection{The group~$\s(G)$}

Let us use the notation~$\theta $ and~$\delta $ for the permutations induced on~$\pc$ by the automorphisms denoted by the same symbols in~$Aut(\gb)$. They generate a subgroup~$\langle \theta, \delta \rangle$ in~$S(\pc)$, the symmetric group of the set~$\pc$. One can check the identities~$\theta ^2 = 1$, $\delta^2= 1$, $\delta \theta \delta = \theta \delta \theta $, and it follows that~$\langle \theta , \delta \rangle$ is a homomorphic image of~$S_3$.

We define~$\s(G)$ to be the subgroup of~$S(\pc)$ of those permutations that commute with the action of~$Out(G) \times \langle \theta, \delta \rangle$, and preserve the~$G \times G$-orbits (bearing remark~\ref{rmk-orbits} in mind). Thus~$\s(G)$ is the intersection of the centralizer of a certain subgroup on the one hand, and a Young subgroup of~$S(\pc)$ on the other hand. (By ``Young subgroup'' we mean the product of symmetric groups associated to a partition of a set, here corresponding to the~$G \times G$-orbits.)

For any group~$G$, we have a map~$\gtz(G) \longrightarrow \s(G)$, by the lemmas just established. The rest of this section is dedicated to the proof of:

\begin{thm} \label{thm-main-simple}
When~$G$ is simple and non-abelian, the map 
\[ \gtz(G) \longrightarrow \s(G)  \]
is an isomorphism.
\end{thm}

Recall that we have a model of~$\gb$ as the subgroup of the cartesian product~$G^r$ generated by~$x = (x_1, x_2, \ldots, x_r)$ and~$y = (y_1, y_2, \ldots , y_r)$. The map~$p_i$ is then just the projection onto the~$i$-th factor.  In~\cite{gareth} one finds a proof of the following

\begin{lem}
Let~$G$ be a nonabelian, simple finite group. Then 
\begin{enumerate}
\item The group~$\gb$ is all of~$G^r$.
\item The normal subgroups of~$G^r$ are those of the form~$\prod_I G_i$ for some~$I \subset \{ 1, \ldots, r \}$ (where~$G_i$ is the~$i$-th embedded copy of~$G$ in~$G^r$).

\item As a result, the maximal, proper normal subgroups of~$G^r$ are those of the form~$\prod_{i \ne j} G_i $ for some~$j$. This is precisely~$K_j$.

\end{enumerate}
\end{lem} 

In the rest of this section~$G$ will always be nonabelian and simple, as well as finite. Let us add :

\begin{lem} The automorphisms of~$G^r$ are as follows:
\begin{enumerate}
\item $Aut(G^r) \cong Aut(G) \wr S_r$.
\item $Out(G^r) \cong Out(G) \wr S_r$. 
\item The action of~$Out(\gb)$ on~$\pc$ is faithful.
\end{enumerate}
\end{lem}

\begin{proof}
Considering the action on the~$K_j$'s, we obtain a map~$Aut(G^r) \to S_r$ which is clearly split surjective. Now suppose~$\phi$ is an automorphism of~$G^r$ preserving all the~$K_j$'s. By taking intersections, we see that~$\phi$ preserves all the normal subgroups of~$G^r$, including~$G_1$ and~$K_1 \cong G^{r-1}$, these two satisfying~$G_1 \times K_1 = G^r$. By induction, it is immediate that~$\phi$ is of the form~$\alpha_1 \times \cdots \times \alpha_r$. This proves (1).

An inner automorphism of~$G^r$ is the direct product of inner automorphisms of each~$G_i$, so we have also (2).

Now suppose~$\phi \in Aut(\gb)$ acts trivially on~$\pc$. Then it must also act trivially on~$\pairs/Aut(G)$ (the map~$\pc \to \pairs/Aut(G)$ is equivariant for this action). It follows that~$\sigma (\phi)$ is the trivial permutation of~$S_r$. The map considered in (2) then sends~$\phi$ to~$(\phi_1, \phi_2, \ldots \phi_r) \in Out(G)^r$, in a notation which is consistent with our earlier use of~$\phi_i$. As the action on~$\pc$ is trivial, we see that~$\phi_i$ must be inner (that is, it represents the trivial element in~$Out(G)$), so (2) implies that~$\phi$ is itself inner. This concludes the proof of the lemma.
\end{proof}

If we combine (3) of the lemma with proposition~\ref{prop-map-to-centraliser}, we see that~$Out(\gb)$ injects into~$C_S(Out(G))$, the centralizer of~$Out(G)$ in~$S= S(\pc)$. However since the action of~$Out(G)$ on~$\pc$ is free with~$r$ orbits, its centralizer is itself a wreath product $Out(G) \wr S_r$. Comparing orders, we conclude that~$Out(\gb)$ maps isomorphically onto~$C_S(Out(G))$ via the action on~$\pc$. 

It remains to check that the conditions defining~$\gtz(G)$ as a subgroup of~$Out(\gb)$ correspond to what is stated in the theorem. This is immediate for the commutation with~$\theta $ and~$\delta $. Lemma~\ref{lem-gtz-preserves-GxG} does half the remaining work, by showing that the elements of~$\gtz(G)$ must preserve the~$G \times G$-orbits in~$\pc$. The proof will be concluded by establishing the converse.

Indeed, if the action of~$\phi \in Aut(\gb)$ is such that~$[\phi_i^{-1} (x_{\sigma (i)}), \phi_i^{-1}(y_{\sigma (i)})]_c$ is in the~$G\times G$-orbits of~$[x_i, y_i]_c$ for all~$1 \le i \le r$ then~$x_{\sigma (i)}$ and~$\phi_i(x_i)$ are conjugate; in other words~$p_{\sigma (i)}(x)$ and~$p_{\sigma (i)}( \phi(x))$ are conjugate. If we recall that~$\gb = G^r$ and each~$p_j$ is just the projection onto the~$j$-th factor, then we see immediately that~$x$ and~$\phi(x)$ are conjugate in~$G^r$, so in~$\gb$. Likewise for~$y$. This concludes the proof of theorem~\ref{thm-main-simple}.

\subsection{Properties of~$\s(G)$}

We write~$S = S(\pc)$ and~$H = Out(G) \times \langle \theta , \delta \rangle$, while~$Y$ is the Young subgroup of those permutations in~$S$ which preserve the~$G \times G$-orbits on~$\pc$. We have~$\s(G) = C_S(H) \cap Y$.

\begin{prop} \label{prop-prod-wreath-prods}
The group~$\s(G)$ is a product of wreath products~$E_k \wr S_{r_k}$ where~$\sum_k r_k = r$ and~$E_k$ is a subquotient of~$H$.

Moreover each integer~$r_k$ satisfies
\[ r_k \le \frac{|Z| m^2} {|G|}  \]
where~$m$ is the size of the largest conjugacy class in~$G$, and~$Z$ is the centre of~$G$.
\end{prop}

Of course we will mostly use this proposition when~$G$ is simple and non-abelian, so that~$|Z|=1$. 

\begin{proof}
We number arbitrarily the ``orbits''~$P_1, P_2, \ldots $ of~$G \times G$ on~$\pc$. The use of quotes here refers to remark~\ref{rmk-orbits}. Other orbits in this proof are genuine.

Every orbit~$X$ of~$H$ has an ordered ``partition'' into the subsets~$X ^{(1)} =X \cap P_1, \, X ^{(2)} = X \cap P_2, \ldots $ some of which may be empty. Call two of these~$H$-orbits~$X_1$ and~$X_2$ {\em equivalent} when there is an~$H$-equivariant bijection~$ X_1 \to X_2$ mapping~$X_1 ^{(i)}$ onto~$X_2 ^{(i)}$ for each~$i$. Finally a {\em block} is a subset of~$\pc$ obtained as the union of the~$H$-orbits inside one equivalence class. 

The image of an~$H$-orbit under an element of~$\s(G)$ is another~$H$-orbit which is equivalent to the original one. It follows that~$\s(G)$ preserves the blocks, as does~$H$. Moreover this allows for a decomposition of~$\s(G)$ as a direct product of groups, one for each block: namely, if~$B$ is a block, define~$\s(G)_B = C_{S(B)}(H) \cap Y_B$ where~$Y_B$ is the Young subgroup corresponding to the partition of~$B$ by the subsets~$B \cap P_i$; then~$\s(G)$ is the direct product of the various groups~$\s(G)_B$.

Suppose that the~$H$-orbits in the block~$B$ are~$X_1, \ldots , X_s$. These are permuted by~$\s(G)$, or~$\s(G)_B$, yielding a homomorphism~$\s(G)_B \to S_s$ which is easily seen to be split surjective. The kernel of this homomorphism is a direct product of~$s$ copies of the group~$E$ of self-equivalences of~$X_1$ (in the above sense). The latter is a subgroup of the automorphism group of~$X_1$ as an~$H$-set; if~$X_1 \cong H/K$ for some subgroup~$K$ of~$H$, then this automorphism group is~$N_H(K)/K$. This completes the description of~$\s(G)_B$ as a wreath product~$E \wr S_s$ where~$E$ is a subquotient of~$H$.

There remains to prove the bound on~$r_k$. If~$X_1$ is an~$H$-orbit, then an~$H$-equivariant bijection~$X_1 \to X_2$ is entirely determined by the image of a single point~$p \in X_1$; if this bijection is to afford an equivalence between~$X_1$ and~$X_2$, then this image must be taken in the~$G\times G$-''orbit'' of~$p$. As a result, there are no more orbits equivalent to~$X_1$ than elements in the largest~$G \times G$-''orbit''. A~$G \times G$-''orbit'' on~$\pairs$ has size~$\le m^2$; on the other hand as~$G/Z$ acts freely on~$\pairs$, the fibres of the map~$\pairs \to \pc$ have size~$|G|/|Z|$. Thus we see that a~$G\times G$-''orbit'' on~$\pc$ has size~$\le |Z| m^2 / |G|$.
\end{proof}

Keeping in mind that~$\langle \theta, \delta \rangle$ is a homomorphic image of~$S_3$, we draw:

\begin{coro}
A simple factor occuring in~$\s(G)$ must be isomorphic to either: \begin{itemize}
\item $C_2$,
\item $C_3$,
\item a subquotient of~$Out(G)$,
\item an alternating group~$A_s$ where~$s \le \frac{|Z| m^2} {|G|}$.
\end{itemize}
\end{coro}

Combining this with theorem~\ref{thm-main-simple} yields a description of the possible simple factors in~$\gtz(G)$ when~$G$ is non-abelian and simple (namely, those in the corollary). Also using that~$\gt(G) / \gtz(G)$ is abelian (lemma~\ref{lem-gtz-abelian}), we draw:

\begin{coro}
Let~$G$ be non-abelian and simple. Then a simple factor occuring in~$\gt(G)$ must be isomorphic to either: \begin{itemize}
\item a cyclic group,
\item a subquotient of~$Out(G)$,
\item an alternating group~$A_s$ where~$s \le \frac{m^2} {|G|}$.
\end{itemize}

\end{coro} 

Note that the classification of finite simple groups implies by inspection that the group~$Out(G)$ is always solvable, so if one accepts this result then we conclude that the list of simple factors reduces to cyclic and alternating groups.

In any case, we can consider those non-abelian simple groups such that~$Out(G)$ has order~$1$,~$2$ or~$4$: this includes the alternating groups for~$n \ge 5$, almost all Chevalley groups over fields of prime order, and all 26 sporadic groups. The results above show that~$\gtz(G)$ can only have, as simple factors, the groups~$C_2$ and~$C_3$ as well as some alternating groups. In~$\gt(G)$ one may encounter further cyclic groups.

\subsection{A complete example} \label{subsec-example-psl}

Take~$G = PSL_3(\f_2)$, a simple group of order~$168$ with~$Out(G) = C_2 = \langle \alpha \rangle$. The following information is obtained with the help of GAP.

There are~$114$ elements in~$\pc$; in the following some arbitrary numbering is used for them. Looking at the action of~$G \times G$ we obtain the following partition of~$\{ 1, \ldots, 114 \}$ :
\begin{multline*}
\{ 1, 10, 27, 28, 96, 106 \}, \{ 2, 11, 52, 57, 82, 86 \}, \{ 3, 12 \}, \{ 62, 63 \}, \{ 107, 109 \}\\ \{ 4, 13, 29, 49, 51, 54, 56, 70, 101, 102 \}, \{ 5, 6, 31, 32, 71, 72 \}, \{ 7, 34, 75 \}, \{ 8, 84, 97 \}, \\\{ 9, 36, 53, 73, 83, 85 \}, \{ 14, 16, 38, 60, 91, 114 \}, \{ 15, 61, 76 \}, \{ 17, 98, 110 \},\\ \{ 18, 40, 58, 59, 78, 99 \}, \{ 19, 21 \}, \{ 20, 41, 43 \}, \{ 22, 42, 64 \}, \{ 39, 77, 90 \},\\\{ 23, 25, 44, 48, 67, 79, 87, 103, 108, 111 \}, \{ 24, 45, 68, 94, 105, 112 \}, \{ 33, 35, 74 \},\\ \{ 26, 46, 69, 89, 93, 104 \}, \{ 30, 37, 50, 55, 100, 113 \}, \{ 47, 65, 66, 80, 81, 88, 92, 95 \}.
\end{multline*}
There are 4 subsets of size 2; 8 of size 3; 9 of size 6; one of size 8 and two of size 10. Thus~$Y \cong C_2^4 \times S_3^8 \times S_6^9 \times S_8 \times S_{10}^2$.

Then we compute \begin{itemize}

\item the permutation induced by~$\alpha $ : (1,11) (2,10) (3,12) (4,101) (5,16) (6,14) (7,15) (8,17) (9,18) (13,102) (19,21) (20, 22) (23,48) (24,26) (25,103) (27,52) (28,57) (29,70) (30,100) (31,60) (32,38) (33, 77) (34,61) (35,39) (36,59) (37,113) (40,85) (41,42) (43,64) (44,87) (45,104) (46, 94) (47,88) (49,56) (50,55) (51,54) (53,58) (62,63) (65,80) (66,95) (67,111) (68, 93) (69,105) (71,114) (72,91) (73,78) (74,90) (75,76) (79,108) (81,92) (82,96) (83, 99) (84,110) (86,106) (89,112) (97,98) (107,109).

\item the permutation induced by~$\theta $ : (1,5) (2,14) (3,19) (4,23) (6,10) (7,15) (8,20) (9,24) (11,16) (12,21) (13,25) (17, 22) (18,26) (27,31) (28,72) (29,67) (30,100) (32,96) (34,61) (36,68) (38,82) (40, 69) (41,97) (42,98) (43,84) (44,49) (45,53) (46,99) (47,88) (48,101) (50,55) (51, 79) (52,60) (54,108) (56,87) (57,91) (58,104) (59,93) (62,107) (63,109) (64,110) (66, 95) (70,111) (71,106) (73,112) (75,76) (78,89) (81,92) (83,94) (85,105) (86,114) (102, 103).

\item the permutation induced by~$\delta $ : (2,3) (4,106) (5,35) (6,8) (7,85) (10,12) (13,100) (14,17) (15,40) (16,39) (19,22) (20,21) (23,68) (24,65) (25,103) (26,80) (27,52) (28,70) (29,57) (30,102) (31,75) (34, 73) (36,71) (37,96) (41,63) (42,62) (44,87) (45,111) (46,107) (47,88) (48,93) (49, 55) (50,56) (51,54) (53,72) (58,91) (59,114) (60,76) (61,78) (66,108) (67,104) (69, 105) (74,84) (79,95) (81,92) (82,113) (83,97) (86,101) (89,112) (90,110) (94,109) (98, 99).

\end{itemize}

We can then ask GAP to compute~$\gtz(G)$ as the intersection of~$Y$ and the centralizer of the three permutations above. We find that~$\gtz(G)$ has order 512; its centre~$Z$ is elementary abelian of order~$32$; and~$\gtz(G)/Z$ is elementary abelian of order~$16$. In fact, finer use of GAP as described below allows us to improve this very last step of the computation, showing that~$\gtz(G) \cong C_2^3 \times D_8^2$, where~$D_8$ is the dihedral group of order~$8$.

\section{Computing explicitly} \label{sec-gap}

In this Section we provide details on the use of the computer algebra system GAP in order to apply our results about~$\s(G)$. No doubt many readers who are not computer enclined will wish to skip most of this, and we encourage them to browse the results themselves in~\S\ref{subsec-simple} and~\S\ref{subsec-p-groups}.

We have chosen to give the explanation in a mathematical discourse interspersed with GAP commands. We feel that the readers having little familiarity with computational group theory will be able to understand what follows, while an opportinity is given to get a sense of ``what is feasible with just one command'' (and by contrast, what requires more effort). On the other hand, we find it useful to indicate some relevant GAP commands to those readers who will wish to implement their own calculations. 

\subsection{Computing~$\gt_{1}(G)$}

The first task is to construct~$\gb$ given~$G$, and it is straightforward. We provide some details solely with the purpose of indicating some GAP functions to the reader. One builds the automorphism group of~$G$ using \gap{AutomorphismGroup(G)}, then converts its generators into automorphisms of~$G \times G$, see \gap{DirectProduct(G, G)} and also the function \gap{GroupHomomorphismByImages}.

Having thus constructed the group \gap{A} of these automorphisms of~$G \times G$, one appeals to \gap{OrbitsDomain(A, GG)} (where \gap{GG} is~$G \times G$) to find the orbits. For each orbit \gap{orb}, pick a representative \gap{orb[1]} (which is an element of~$G \times G$), extract the two elements \gap{x} and \gap{y} comprising the pair, and check whether we have the equality \gap{Subgroup(G, [x, y]) = G}. If not, discard the orbit.

Picking representatives in the~$r$ remaining orbits, one constructs~$\gb$ as a subgroup of~$G^r$, using \gap{DirectProduct} again.

As for~$\gtz(G)$, we will rely on the next lemma.

\begin{lem}
Let~$C_x$ and~$C_y$ be the centralizers of the canonical generators~$x$ and~$y$ of~$\gb$, respectively. To each element~$\phi \in \gtz(G)$ we may associate a unique double coset~$D \in C_x \bs \gb / C_y$. In fact, if~$\phi$ is induced by the automorphism~$\tilde \phi$ of~$\gb$ satisfying~$x \mapsto x^f$ and~$y \mapsto y$, then~$D$ is the double coset of~$f$.

Moreover, in the same notation, we have~$1 \in C_x \, f \theta (f) \, C_y^{\theta (f)}$. Conversely if~$\tilde \phi \in Aut(\gb)$ satisfies~$x \mapsto x^f$ and~$y \mapsto y$ for some~$f$ such that~$1 \in C_x \, f \theta (f) \, C_y^{\theta (f)}$, then the induced~$\phi \in Out(\gb)$ commutes with~$\theta $.
\end{lem}

\begin{proof}
By definition~$\phi$ can be induced by such an automorphism. If~$f_1$ and~$f_2$ are both possible choices for~$f$, yielding~$\tilde \phi_1$ and~$\tilde \phi_2$ both inducing~$\phi$, then we see that~$\tilde \phi_2 = c_t \circ \tilde \phi_1$ where~$c_t$ is conjugation by~$t$. It follows that~$y^t = y$ so~$t \in C_y$, and that $x^{f_2} = x^{f_1 t}$ so~$s:= f_1 t f_2^{-1} \in C_x$. In the end~$f_1 = s f_2 t^{-1}$ so~$f_1$ and~$f_2$ are in the same double coset. The converse is obvious.

We turn to the last statement, which follows from the fact that~$\tilde \phi \circ \theta $ and~$\theta \circ \tilde \phi$ must differ by an inner automorphism, by definition of~$\gtz(G)$. So there must exist an element~$t$ such that~$y^t = y^{\theta (f)}$ and~$x^{ft} = x$. We see that~$s = ft \in C_x$ and~$u = t \theta (f)^{-1} \in C_y$, so that~$1 = s^{-1} f u \theta (f) \in C_x f C_y \theta (f) =  C_x \, f \theta (f) \, C_y^{\theta (f)}$. Again the converse is left to the reader.
\end{proof}

This suggest the following method to compute~$\gtz(G)$. First, compute the centralizers \gap{Cx:= Centralizer(GB, x)} and \gap{Cy:= Centralizer(GB, y)}, where \gap{GB} is~$\gb$. Then compute \gap{DoubleCosets(GB, Cx, Cy)} (which is an optimized process in GAP). Now filter the double cosets, by excluding~$C_x f C_y$ if the unit of~$\gb$ is not in \gap{DoubleCoset(Cx, f*fth, Cy\^{}fth)}, where \gap{fth} is~$\theta(f)$. Also exclude~$f$ if~$x^f$ and~$y$ generate a proper subgroup of~$\gb$.

The next step is to go through the remaining double cosets, and with each representative~$f$ define~$\tilde \phi$ with \gap{GroupHomomorphismByImages(GB, GB, [x, y], [x\^{}f, y])} (this is automatically well-defined by (3) of lemma~\ref{lem-basic-props-gb}; however the construction of~$\tilde \phi$ by GAP is surprisingly time-consuming, which is the reason for filtering out as many candidates for~$f$ as possible before reaching this stage).

Finally, keep only those homomorphisms commuting with~$\delta $ in~$Out(\gb)$, which may be checked with \gap{IsInnerAutomorphism(phi*delta*phi\^{}(-1)*delta\^{}(-1))}.

At this point, we have a list of automorphisms of~$\gb$ representing the elements of~$\gtz(G)$ with no repetition; the number of these automorphisms is the order of~$\gtz(G)$. Finding the group structure of~$\gtz(G)$ can be achieved, rather slowly, with the help of the commands \begin{verbatim}
A:= AutomorphismGroup(GB);
int:= InnerAutomorphismsAutomorphismGroup(A);
quo:= NaturalHomomorphismByNormalSubgroup(A, int);
\end{verbatim}

One can then create the list of all the elements \gap{quo(phi)} where \gap{phi} is taken from our list of automorphisms, and ask GAP to describe the group they generate.

\subsection{Computing~$\s(G)$}

This is several orders of magnitude faster than computing~$\gtz(G)$. 

The first step is to construct~$\pc$. For this, one builds~$G \times G$ and the embedded diagonal copy of~$G$ in~$G \times G$, and one appeals to \gap{OrbitsDomain(diagG, GG)}. As above, one filters out an orbit if a representative pair~$(x, y)$ fails to generate all of~$G$. We obtain~$\pc$ as a GAP list, say \gap{pairsconj}. Its length is~$\ell = r \, |Out(G)|$.

Then one must build the orbits of~$G \times G$ on~$\pc$, say stored as a list of lists of indices, those indices refering to \gap{pairsconj}. It is impossible to rely on \gap{OrbitsDomain} for the reasons given in remark~\ref{rmk-orbits}, since we are not dealing with genuine orbits ; instead, for each \gap{pairorbit} taken in \gap{pairsconj}, we take the representative \gap{pair:= pairorbit[1]}, then get its \gap{C:= ConjugacyClass(GG, pair)} where \gap{GG} is~$G \times G$. Then we run through all the indices, and those corresponding to elements of~$\pc$ which happen to be in \gap{C} we group together. After this has been done for each \gap{pairorbit}, we have a list of lists of indices partitioning the set~$\{ 1, \ldots, \ell \}$; for example in \S\ref{subsec-example-psl} we gave the corresponding partition of~$\{ 1 \ldots 114 \}$ when~$G = PSL_3(\f_2)$. The corresponding Young subgroup \gap{Y} may be constructed easily, but we will argue that it is not the best way to go.

Before turning to this though, we mention that we must construct two further permutations of~$\{ 1, \ldots, \ell \}$ corresponding to~$\theta $ and~$\delta $, and for this we follow lemma~\ref{lem-theta-no-surprises} and lemma~\ref{lem-delta-no-surprises}. We then do the same for each generator of~$Out(G)$ which is not inner, and compute the corresponding permutations. We let GAP know that we call \gap{H} the subgroup of~$S_\ell$ generated by all these elements.  

It is possible at this stage to ask GAP to compute \begin{verbatim}
Intersection( Centralizer(SymmetricGroup(ell), H), Y);
\end{verbatim}
but except in very small examples, the calculation will simply take too long (and will likely exhaust the available memory). 

Instead, we partially implement the ideas of proposition~\ref{prop-prod-wreath-prods} and its proof. Let us use the notation introduced in that proof. First we compute the orbits of~$H$ with \gap{OrbitsDomain(H, [1..ell])}. Then we group these orbits according to a looser equivalence relation than the one used in the proof of~\ref{prop-prod-wreath-prods}: we call two orbits~$X_1$  and~$X_2$ equivalent if (i) they are isomorphic as~$H$-sets, which in practice is checked by verifying whether the corresponding stabilizers are conjugate in~$H$, {\em cf} \gap{Stabilizer(H, orbit[1])} and \gap{ConjugacyClassesSubgroups(H)} which together allow to associate to each orbit the position of the conjugacy class of the stabilizer in some numbering; and (ii) the {\em cardinality} of~$X_1 \cap P_i$ is equal to the cardinality of~$X_2 \cap P_i$, for each index~$i$. The union of the orbits in one equivalence class we call a {\em packet}, and each packet is a union of the ``blocks'' defined in the aforementioned proof.

Now each packet is~$H$-invariant, and~$\s(G)$ splits as a direct product corresponding to the packets, which we see by arguing as we did in the proof of~\ref{prop-prod-wreath-prods} with blocks. We can compute the image~$H'$ of~$H$ in the symmetric group of each packet by applying \gap{RestrictedPermNC(g, packet)} to each generator \gap{g} of~$H$. Also, the Young subgroup~$Y'$ corresponding to the intersections of the~$P_i$'s with the packet is readily created in GAP. 

It is now possible to ask directly for the computation of the centralizer of~$H'$, and its intersection with~$Y'$. The product of all the resulting groups, for all packets, is~$\s(G)$. For each factor in this product, we can prompt GAP for the composition series, see for example \gap{DisplayCompositionSeries(factor)}. (Finer information, such as \gap{StructureDescription(factor)}, can still take a very long time).

\subsection{Simple groups of small order} \label{subsec-simple}

We shall give information on~$\gtz(G)$ for~$12$ simple groups of small size.
We start with the~$PSL_2$ family (recall that~$PSL_2(\f_4) \cong PSL_2(\f_5) \cong A_5$, $PSL_2(\f_7) \cong PSL_3(\f_2)$ and $PSL_2(\f_9) \cong A_6$). Write~$D_8$ for the dihedral group of order~$8$.

\begin{thm} \label{thm-conj-psl2q} We have:

\begin{itemize}
\item (Order 60) $\gtz(PSL_2(\f_4))$ is trivial.
\item (Order 168) $\gtz(PSL_2(\f_7)) \cong C_2^3 \times D_8^2$.
\item (Order 360) $\gtz(PSL_2(\f_9)) \cong C_2^{12} \times D_8$.
\item (Order 504) $\gtz(PSL_2(\f_8))$ is trivial.
\item (Order 660) $\gtz(PSL_2(\f_{11})) \cong C_2^{27} \times D_8^7$.
\item (Order 1092) $\gtz(PSL_2(\f_{13})) \cong C_2^{54} \times D_8^{17}$.
\item (Order 2448) $\gtz(PSL_2(\f_{17})) \cong C_2^{104} \times D_8^{50}$.
\item (Order 3420) $\gtz(PSL_2(\f_{19})) \cong C_2^{133} \times D_8^{74}$.
\item (Order 4080) $\gtz(PSL_2(\f_{16}))$ is trivial.
\end{itemize}
\end{thm}

It seems tempting to conjecture that~$\gtz(PSL_2(\f_q)) \cong C_2^a \times D_8^b$, with~$a= b= 0$ when~$q$ is a power of~$2$. Let us now turn to the simple group of order 5616:

\begin{thm}
The group~$\gtz(PSL_3(\f_3))$ is isomorphic to 
\[ C_2 ^{26} \times D_8^6 \times S_3^4 \times S_4^{21} \times S_6^{12} \times S_7^6 \times S_8^3 \times S_9^{11} \times A^3 \times B^8 \times C^5 \]
where 
\[ A = (((C_2 \times C_2 \times C_2 \times C_2 \times C_2) \rtimes A_6) \rtimes C_2) \rtimes C_2  \, , \]
\[ B = ((((C_2 \times D_8) \rtimes C_2) \rtimes C_3) \rtimes C_2) \rtimes C_2 \, ,   \]
and 
\[ C=  (((C_2 \times C_2 \times C_2 \times C_2) \rtimes A_5) \rtimes C_2) \, .   \]
\end{thm}

Two more simple groups from our intention remain. It has not been possible to obtain a complete description of~$\gtz(G)$ for these (in a reasonable amount of time), no doubt because of the appearance of simple factors of the form~$A_s$ for~$s$ large (18 and above). At least we have been able to find the corresponding simple factors.

\begin{thm}
The direct product of the simple factors of~$\gtz(A_7)$ is 
\[ C_2^{152} \times C_3^{15} \times A_5^3 \times A_6^3 \times A_7 \times A_8^2 \times A_{10} \times A_{18} \, .   \]
%
\end{thm}

\begin{thm}
The direct product of the simple factors of~$\gtz(M_{11})$ is 
\begin{multline*}
 C_2^{465} \times C_3 ^{46} \times A_5^{10} \times A_6^9 \times A_7^{10} \times A_8^4 \times A_9^4 \times A_{10}^5 \times A_{11}^5 \times A_{12} \times A_{14}^2 \times A_{15}^4 \times A_{16} \times A_{17}^3 \times A_{18}^{12}\\ \times A_{19} \times A_{20}^2 \times A_{23} \times A_{28} \times A_{31} \times A_{33}^2 \, . 
\end{multline*}
\end{thm}

\subsection{$p$-groups} \label{subsec-p-groups}

We are going to give information on the orders of~$\gtz(G)$ and~$\s(G)$ for a few~$2$-groups. We offer the following table:

\bigskip
\begin{tabular}{|r|r|r|r|r|}
\hline
$n$ & number of groups & $\max |\gtz(G)|$ & $\max |\s(G)|$ & same order (1 or 2) \\
\hline 
3 & 2 & 1 & 1 & 2 (2) \\
4 & 6 & 1 & 1 & 6 (6) \\
5 & 17 & 2 & 1 & 14 (14) \\
6 & 50 & 4 & 2 & 40 (40) \\
7 & 159 & 4 & 16 & 111 (109) \\
\hline
\end{tabular}
\bigskip

The first column contains a number~$n$, indicating that the row is about groups of order~$2^n$. Among all these, we select those which can be generated by~$2$ elements, and which are non-abelian (otherwise~$\gtz(G)$ is not interesting). The next column gives the number of groups which we have kept. The largest order for~$\gtz(G)$, when~$G$ runs among those groups, is recorded in the next column, followed by the maximum for~$|\s(G)|$. Finally, we give the number of groups for which~$| \gtz(G)| = | \s(G) |$ followed in parenthesis by the number of groups for which this common order is $1$ or $2$ (so that $\gtz(G)$ and~$\s(G)$ are at least abstractly isomorphic in these cases).

It may be difficult to comment sensibly on this small-scale analysis. We venture to say that~$\s(G)$ is ``quite often'' abstractly isomorphic to~$\gtz(G)$ (though not always), and that the order of~$\gtz(G)$ seem to grow very slowly with that of~$G$ (attempts to prove theoretical bounds on the order of~$\gtz(G)$ have produced very bad results, and there seems to be a phenomenon to understand here).





\section{Dessins d'enfants} \label{sec-dessins}

Keeping the notation introduced before, we have seen that there is a natural action of~$\gt(G)$ on the set~$\pairs/ Aut(G)$. However, in the process of constructing a map from~$\gtz(G)$ to~$\s(G)$, we have also defined an action of~$\gt(G)$ (in fact, of~$Out(\gb)$) on~$\pc$, the set of pairs in~$\pairs$ up to conjugation. This seems a little {\em ad hoc}. In this section we provide a partial ``explanation'', showing that at least for~$\gal$ (a close cousin of~$\gt(G)$) there are very good reasons for an action to exist on~$\pc$.

\subsection{The category of dessins}

Dessins d'enfants are essentially bipartite graphs drawn on compact, oriented surfaces, forming the~$1$-skeleton of a CW-complex structure (in particular, the complement of the graph is homeomorphic to a disjoint union of open discs). For precise definitions, and for all statements about dessins to follow, we refer to~\cite{pedro}. 

Dessins form a category~$\cat$, and the first striking results of the theory are equivalences between~$\cat$ and various other categories. To name but the most important one: algebraic curves over~$\C$ with appropriate ramification, étale algebras over~$\C(x)$ with a ramification condition, the other two categories obtained by replacing~$\C$ by~$\qb$, and the category of finite sets with a right action of~$F_2$, to be denoted simply by~$\sets$ (our usual notation for the generators of~$F_2$ being~$x, y$).

Some dessins are called {\em regular}. Many definitions are possible; in~$\sets$, an object is regular if and only it is isomorphic to one of the following particular form: take a finite group~$G$ with two distinguished generators~$x$ and~$y$, and let~$F_2$ act on~$G$ on the right via the canonical homomorphism~$F_2 \to G$, using right multiplication. The regular dessins~$(G, x, y)$ and~$(G', x', y')$ (as we will denote them) are isomorphic in~$\sets$ if and only if there is a group isomorphism~$G \to G'$ with~$x \mapsto x'$, $y \mapsto y'$, as the reader may check.

The automorphism group of the regular dessin~$(G, x, y)$ is then~$G$ itself, acting on itself by left multiplication. As a result of this discussion, we see that the set of isomorphism classes of regular dessins with automorphism group isomorphic to a fixed group~$G$ is in bijection with~$\pairs/Aut(G)$, the set of pairs of elements generating~$G$ under the action of~$Aut(G)$. 

Thus we can rephrase the fact that~$\gt(G)$ acts on~$\pairs/Aut(G)$ by saying that~$\gt(G)$ acts on the isomorphism classes of regular dessins with automorphism group~$G$. What is more, one can define an action of~$\gt = \lim_G \gt(G)$ on the isomorphism classes of all dessins, regular or not ({\em loc cit}). 

The other fundamental result is the fact that there is also a natural action of~$\gal$ on isomorphism classes of dessins. This action factorizes through the map $\Phi \colon \gal \longrightarrow  \gt$ (and is indeed instrumental in the very definition of~$\Phi$); the classical result due to Belyi and Grothendieck that the action of~$\gal$ is faithful implies then the injectivity of this homomorphism. In turn one can show that~$\gt$ also acts faithfully, and even that the action on regular dessins is already faithful (theorem 5.7 in~\cite{pedro}).

\subsection{$\Gamma $-dessins}

In~\cite{pedro} we have defined a certain category~$\etq$, whose objects are étale algebras over~$\qb(x)$ satisfying a certain ramification property, and we have built an equivalence of categories between~$\cat$ and~$\etq$. The category~$\etq$ is the one to use in order to define the action of~$\gal$ on isomorphism classes of dessins.

What is more, we proved in {\em loc cit} that each~$\lambda \in \gal$ actually defines a {\em functor}~$F_\lambda \colon \etq \longrightarrow \etq$, inducing the action. This gives more precise information, as we proceed to show. We shall work in a very general context, not for the sheer pleasure of writing abstract nonsense, but out of necessity: the equivalence between~$\cat$ (or the very practical~$\sets$) and~$\etq$ is given by a zig-zag of explicit functors, whose inverses we know very little about beyond the fact that their existence is garanteed by the axiom of choice. This prevents us from being more direct and concrete.

So let~$\ct$ be any category at all. Given a group~$\Gamma $, we can define the category~$\Gamma \ct$ whose objects are the pairs~$(X, \rho )$ where~$X$ is an object of~$\ct$ and~$\rho \colon \Gamma \to Aut_\ct(X)$ is a group homomorphism. (In the sequel we shall write~$Aut(X)$ rather than~$Aut_\ct(X)$ when no confusion can arise.) The morphisms $(X, \rho ) \to (Y, \rho ')$ in~$\Gamma \ct$ are those morphisms~$f \colon X \to Y$ in~$\ct$ which are equivariant in the sense that the following diagram commutes, for any~$g \in \Gamma $: 
\[ \begin{CD}
X @>{\rho (g)}>> X \\
@VfVV        @VVfV \\
Y @>{\rho' (g)}>> Y
\end{CD}
  \]

Now let~$F$ be a self-equivalence of~$\ct$. The operation~$[X] \mapsto [F(X)]$ gives a permutation of the set of equivalence classes of objects in~$\ct$, where we have written~$[X]$ for the class of~$X$. 

However, more is true. Simply assuming that~$F$ is a functor from~$\ct$ to itself, there is an induced homomorphism~$a_X \colon Aut(X) \to Aut(F(X))$, and we can employ it to construct a self-functor~$\tilde F$ of~$\Gamma \ct$. On objects this is defined as~$\tilde F(X, \rho ) = (F(X), a_X \circ \rho )$, while on morphisms~$\tilde F$ is simply the restriction of~$F$. One checks readily that~$\tilde F$ is indeed a functor, that~$\widetilde{F \circ G} = \tilde F \circ \tilde G$, and that~$\tilde F$ is the identity of~$\Gamma \ct$ if~$F$ is the identity of~$\ct$. In particular, if~$F$ is a self-equivalence of~$\ct$, then~$\tilde F$ is a self-equivalence of~$\Gamma \ct$. 

The comments we have made on~$F$ then apply to~$\tilde F$: there is an induced permutation of the set of isomorphism classes of objects in~$\Gamma \ct$. Since this discussion was conducted purely in the language of categories, it is clear that~$\ct$ can be replaced by any category equivalent to it.

We also point out that the group~$Aut(\Gamma )$ acts on the isomorphism classes of objects, by the rule~$\alpha \cdot (X, \rho ) = (X, \rho \circ \alpha )$ (for~$\alpha \in Aut(\gamma )$). When~$\alpha $ is inner, we see readily that this action is trivial (in fact if~$\alpha $ is conjugation by~$t$, then~$\rho (t) \colon X \to X$ is an isomorphism in~$\Gamma \ct$ between~$(X, \rho )$ and~$(X, \rho \circ \alpha )$). Thus we have an induced action of~$Out(\Gamma )$, and it commutes visibly with the permutation induced by~$\tilde F$.

Coming back to~$\etq$, where we know that the action of~$\lambda \in \gal$ is {\em via} some functor~$F_\lambda $, we conclude:

\begin{prop}
There is an action of~$\gal$ on the set of isomorphim classes of objects in~$\Gamma \cat$, for any group~$\Gamma $. The same holds with $\cat$ replaced by any equivalent category, such as~$\sets$. 

There is also an action of~$Out(\Gamma )$ on the same classes of objects, and the two actions commute.

Moreover, the forgetful functor~$\Gamma \cat \to \cat$ induces a~$\gal$-equivariant map between the sets of isomorphism classes.

Finally, these actions restrict to the set of isomorphim classes of faithful, equivariant dessins.
\end{prop}

The last sentence mentions {\em faithful} equivariant dessins, that is, those objects~$(X, \rho )$ in~$\Gamma \cat$ such that~$\rho $ is injective. The statement holds obviously.

We point out that the group~$\Gamma $ need not be finite here. This is one reason for not calling it~$G$, which in this paper usually denotes a finite group. Still, the reader may complain that in the Abstract and Introduction, we mentioned~$G$-dessins rather than~$\Gamma $-dessins. The point is that we wanted to develop the properties of~$\Gamma $-dessins in general, then consider a regular dessin~$X$ with automorphism group denoted~$G$ as in the rest of the paper, and {\em then} regard it as a~$G$-dessin, that is, putting~$\Gamma  = G$ ultimately.

The reader may wish to skip ahead to \S\ref{subsec-example} where two concrete examples of $\Gamma $-dessins are presented (with~$\Gamma $ a finite, cyclic group). However a little result is required in order to prove that they are not isomorphic to one another, and we turn to this easy point now.

\subsection{$\Gamma $-objects in~$\sets$}

As happens with many other concepts, the category~$\sets$ provides the most clean-cut statements about $\Gamma $-dessins. Recall from the definitions that an object in~$\Gamma \sets$ is a finite set with an action of~$F_2$ on the right, and a commuting action of~$\Gamma $ on the left (since we usually define the composition in~$Aut(X)$ such that it acts on the left on~$X$, in general).

Let us call an object in~$\Gamma \sets$ {\em regular} when it is regular as an object of~$\sets$ (that is, the concept ignores the~$\Gamma $-action).

\begin{prop} \label{prop-equ-obj-sets}
Let~$G$ and~$\Gamma $ be groups. Consider the objects in~$\Gamma \sets$ which  are regular and have automorphism group isomorphic to~$G$. Then the set of isomorphism classes of such objects in~$\Gamma \sets$ is in bijection with the set of triples~$(g, h, \phi)$ where~$\langle g, h \rangle = G$ and~$\phi \colon \Gamma \to G$ is a homomorphism, modulo the relation 
\[ (g, h, \phi_1) \sim (g', h', \phi_2)  \]
which holds (by definition) if and only if there is an automorphism~$\alpha \in Aut(G)$ and~$t \in G$ such that~$\alpha (g)= g'$, $\alpha (h) = h'$, and~$\alpha (\phi_1(\gamma ) ) = t^{-1} \phi_2(\gamma ) t$, for all~$\gamma \in \Gamma $.
\end{prop}

\begin{proof}
It is clear from the definitions that a triple~$(g, h, \phi)$ does define a regular object in~$\Gamma \sets$, and that each such object can be obtained in this way. What needs to be checked is the condition expressing that~$(g, h, \phi_1)$ and~$(g', h', \phi_2)$ yield isomorphic objects in~$\Gamma \sets$. 

So assume that there is~$\alpha_1 \colon G \to G$ giving such an isomorphism; that is, $\alpha $ is a map of sets which is equivariant with respect to both the~$F_2$-actions on the right, and the~$\Gamma $-actions on the left. Let~$t = \alpha_1(1)$ and define~$\alpha_2 \colon G \to G$ by~$\alpha _2(g) = t^{-1}g$. Finally put~$\alpha = \alpha_2 \circ \alpha_1$. We have~$\alpha (1) = 1$, and as we know that~$\alpha $ commutes with the~$F_2$-actions on the right, implying in particular that $\alpha (x) = \alpha (1 \cdot x )= \alpha (1) \cdot x' = x'$ and $\alpha (y) = y'$, it follows easily that~$\alpha $ is in fact a homomorphism. Moreover for~$\gamma \in \Gamma $ we have $\alpha (\phi_1( \gamma   )) = \alpha_2( \alpha_1( \phi_1(\gamma ) \cdot 1)) = \alpha_2( \phi_2(\gamma ) \cdot \alpha_1 (1)) = t^{-1} \phi_2(\gamma ) t$.

It is straightforward to reverse this argument and prove, conversely, that whenever~$\alpha $ and~$t$ exist, the objects defined by~$(g, h, \phi_1)$ and~$(g', h', \phi_2)$ are indeed isomorphic in~$\Gamma \sets$.
\end{proof}

For~$\Gamma = 1$ we find, as we already know, that the set of isomorphism classes of regular dessins with~$G$ as automorphism group is in bijection with~$\pairs/Aut(G)$. Of more interest is the following:

\begin{coro}
For~$\Gamma = G$, the set of isomorphism classes in~$G \sets$ of faithful, regular objects, with~$G$ as automorphism group, is in bijection with~$\pc$.

In particular~$\gal$ acts on~$\pc$ as well as~$\pairs/Aut(G)$, and the map $$\pc \longrightarrow  \pairs/Aut(G)$$ is~$\gal$-equivariant.
\end{coro}

\begin{proof}
Start with an object~$(g, h, \phi)$. Since we consider only faithful objects, and since~$\Gamma = G$ here, the map~$\phi$ is an automorphism. Taking~$\alpha = \phi^{-1}$ in the proposition, we see that the same object is represented by $(g', h', id)$ where~$g' = \phi^{-1}(g)$, $h' = \phi^{-1}(h)$ (and~$id$ is the identity map).

We conclude that the objects under consideration can be represented by  triples of the form~$(g, h, id)$. By the proposition, $(g_1, h_1, id)$ and $(g_2, h_2, id)$ represent isomorphic objects precisely when there is an automorphism~$\alpha $ of the form~$\alpha (\gamma   ) = t^{-1} \gamma t$, that is an inner automorphism, taking~$g_1$ to~$g_2$ and~$h_1$ to~$h_2$. In the end the set of isomorphism classes is precisely~$\pc$.
\end{proof}

Of course this falls short of a proof that~$\gt$, let alone~$\gt(G)$ or~$Out(\gb)$, acts on~$\pc$, but this corollary makes the result much less surprising and much more natural. Also note that our {\em ad hoc} arguments to the effect that~$Out(\gb)$ does act on this set do not guarantee any compatibility with the action of the image of~$\gal \to Out(\gb)$.

\subsection{A complete example; cyclic dessins} \label{subsec-example}

We follow the usual workflow for dessins. We start with an informal picture of a dessin on the sphere: 

\begin{center}
\includegraphics[width=0.25\textwidth]{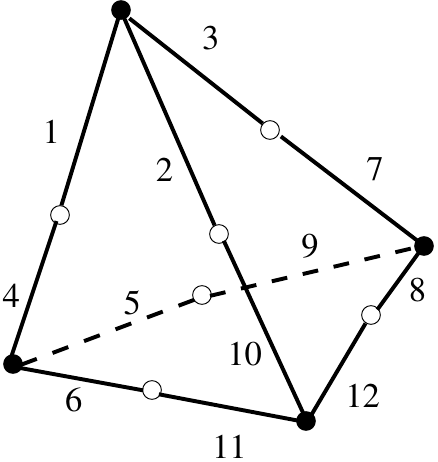}
\end{center}

The theory guarantees that enough information is conveyed in the picture to define an object unambiguously. To proceed with this, we move to~$\sets$. Having numbered the ``darts'' ($=$ edge between a black vertex and a white vertex), we write down the two permutations~$x$ and~$y$ of the set~$\{ 1, \ldots , 12 \}$ which take each dart to the next one in the positive rotation around the incident black, resp.\ white, vertex. These are 
\[ x= (123)(456)(789)(10, 11, 12) ~\textnormal{and}~ y= (14)(2, 10)(37)(59)(6, 11)(8, 12) \, .   \]
Then~$X = (\{ 1, \ldots, 12 \}, x, y)$ is our object in~$\sets$. The subgroup~$G$ of~$S_{12}$ generated by~$x$ and~$y$ has order~$12$, and is in fact isomorphic to~$A_4$; it acts freely and transitively, and it follows that~$X$ is regular. For the rest of the discussion, we identify~$G$ with the set~$\{ 1, \ldots, 12 \}$, by identifying~$g \in G$ with the image of~$1$ under~$g$, and the natural action of~$G$ on this set is by right multiplication.

The automorphism group of~$X$ is given by all the multiplications by elements of~$G$ {\em on the left} on~$\{ 1, \ldots, 12 \}=G$. Thus this group is (isomorphic to) $G$ itself. For example the permutation~$\tilde x$ of the set~$\{ 1, \ldots, 12 \}$ corresponding to the action of~$x$ by left multiplication is the unique automorphism of~$X$ taking~$1$ to~$2$; from the picture we know that this must be the rotation around the black vertex incident with the dart~$1$, that is 
\[ \tilde x = (123)(4, 10, 7)(6, 12, 9)(11, 8, 5) \, .   \]
(This can also be checked by computation.) A similar reasoning gives 
\[ \tilde y = (14)(8, 12)(2, 5)(3, 6)(10, 9)(11, 7) \, .   \]
See Example 3.8 in~\cite{pedro} for more details. 

Let~$C_3 = \langle r \rangle$ be the cyclic group of order~$3$, and let us define $C_3$-dessins with~$X$ as the underlying dessin. Define two homomorphisms~$\phi_1, \phi_2 \colon C_3 \to G$ by~$\phi_1(r) = x$ and~$\phi_2(r)= x^{-1}$. Then~$X_1 = (G, x, y, \phi_1)$ and~$(G, x, y, \phi_2)$ are~$C_3$-dessins. 

Let us show that they are not isomorphic. If they were, by proposition~\ref{prop-equ-obj-sets} there would exist~$\alpha \in Aut(G)$ and~$t \in G$ such that~$\alpha (x) = x$, $\alpha (y) = y$ and~$\alpha (\phi_1(r)) = t^{-1} \phi_2(r) t $. The first two conditions impose~$\alpha = Id$ of course, and the last one reads~$x = t^{-1} x^{-1} t$. However~$x$ and~$x^{-1}$ are not conjugate in~$G$, so~$t$ cannot exist.

To explore the Galois action, we continue with the usual workflow, and compute a Belyi map. A possible choice is 
\[ f(z) = -64 \frac{(z^3+1)^3} {(z^3 - 8)^3 z^3}  \]
(taken from~\cite{zvonkine}). The simple fact that the coefficients of~$f$ are in~$\q$ means that~$X$ is fixed by~$\gal$. However, we shall see that~$X_1$ and~$X_2$ are not.

First we can ask a computer to produce a picture of~$f^{-1}([0,1])$.

\begin{center}
\includegraphics[width=0.7\textwidth]{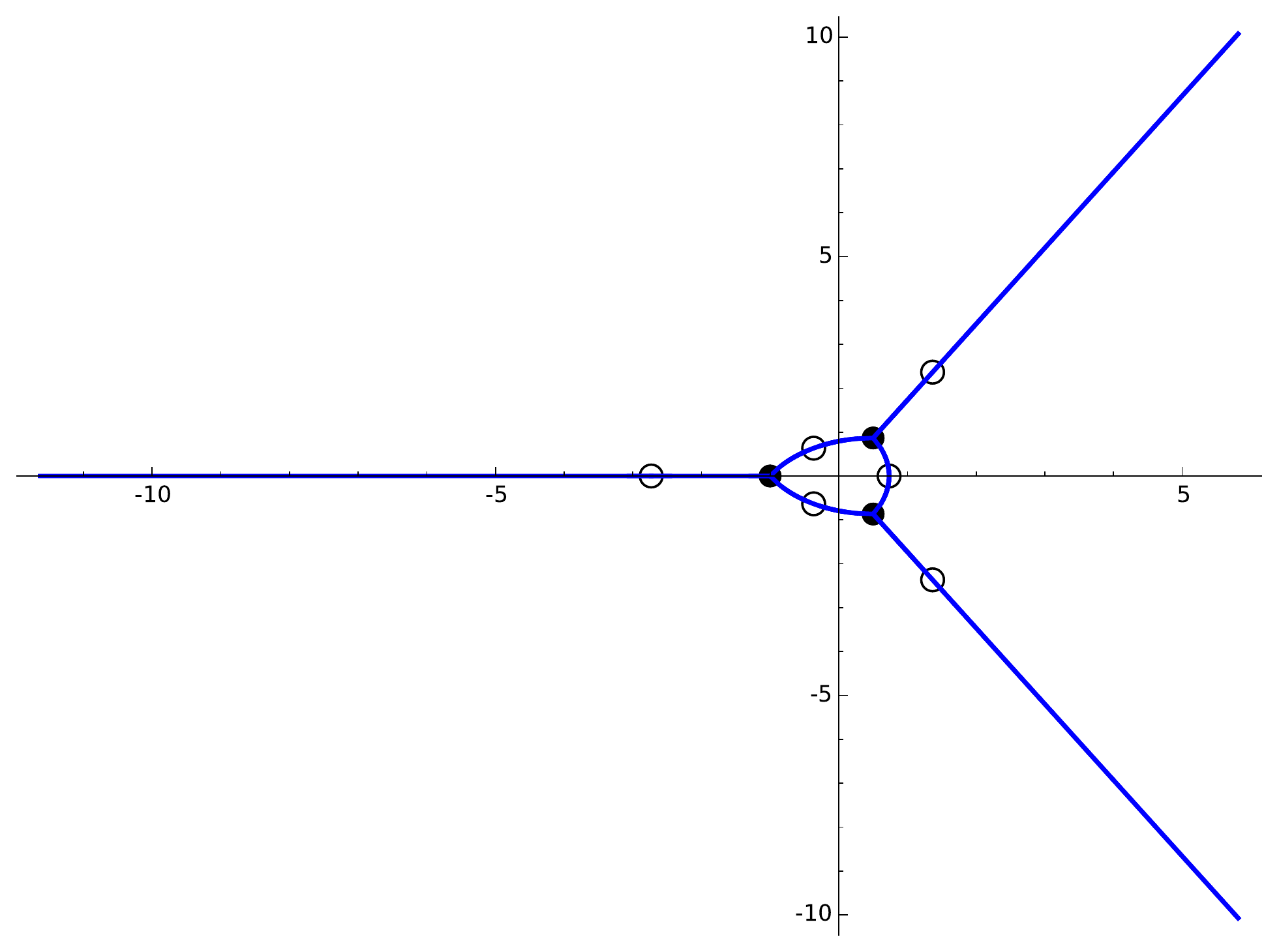}
\end{center}

Here the white vertex on the real axis is~$w= - (1+\sqrt 3)$, a root of~$w^2 + 2w - 2 = 0$. The dart from~$\infty$ to~$w$ we number as~$1$, and we number all the others so that~$x$ and~$y$ are as above.

We know that there must exist Moebius transformations inducing the actions of~$\tilde x$ and~$\tilde y$ as above, and we find easily that they are 
\[ \mu_x \colon z \mapsto j^2 z  \quad\textnormal{and}\quad \mu_y \colon z\mapsto \frac{-z +2~} {~~ z+1}  \]
respectively, where~$j = e^{\frac{2i\pi} {3}}$. (The use of~$j^2$ rather than~$j$ mirrors the fact that we consider the positive rotation around~$\infty$, which is also the clockwise rotation around~$0$.) We can now think of~$X$ as the Belyi pair~$(\p, f)$ and of~$Aut(X)$ as the group generated by~$\mu_x$ and~$\mu_y$. The~$C_3$-dessins~$X_1$ and~$X_2$ are obtained from~$X$ by throwing in the homomorphism~$C_3 \to Aut(X)$ mapping~$r$ to~$\mu_x$ or~$\mu_x^{-1}$.

If~$\lambda \in \gal$ satisfies~$\lambda (j) = j^2$ (and there are such elements!), then the action of~$\lambda $ exchanges~$X_1$ and~$X_2$. [In fact, since the dessin~$X$ is fixed by~$\gal$, we have a homomorphism~$\gal \to Out(Aut(X))$. The non-trivial element of~$Out(A_4) \cong C_2$ sends~$x$ to~$x^{-1}$ and~$y$ to~$y^{-1} = y$.]

In a nutshell: {\em the tetrahedron can be made into a~$C_3$-dessin in two non-isomorphic ways, by picking a rotation of order~$3$ whose axis carries a black vertex and the centre of the opposite face; the two choices are enabled by the two possibilities for the orientation; these are interchanged by the action of~$\gal$.}

\bigskip

More generally, let us call a dessin {\em cyclic} when it is a~$C_n$-equivariant dessin for some~$n$. Starting with a regular dessin~$(G, x, y)$, a cyclic structure on it is simply given by an element of~$G$ of order dividing~$n$ ; the non-isomorphic cyclic structures correspond to the conjugacy classes of such elements. If the dessin is fixed by the Galois group, then~$\gal$ permutes these conjugacy classes.

\section{Concluding comments}

In this paper we have explored the properties of the groups~$\gt(G)$ and~$\s(G)$, from first principles. Motivation for this is twofold: on the one hand, each of these acts on the set of (isomorphism classes of) regular dessins with~$G$ as automorphism group, extending the action of~$\gal$, and this is our finest information about the latter action ; on the other hand, $\gal$ injects in the inverse limit~$\gt$ of all the groups~$\gt(G)$. By contrast one cannot form an inverse limit with the groups~$\s(G)$, but they are much easier to compute with, and in good cases $\s(G)$ agrees with~$\gt(G)$, eg when~$G$ is simple and non-abelian. We have also tried, using the apparatus of~$G$-dessins, to give a conceptual explanation for the existence of~$\s(G)$.

Many people have attempted to classify the generating pairs up to conjugacy or up to isomorphism, within their favorite finite group~$G$. This may or may not have been motivated by the enumeration of dessins. We believe that, once this is done, trying to compute~$\gt(G)$ is a natural idea, that will lead to a wealth of new information.

We shall conclude with a few open problems for the reader. Some of these have been implicitly touched upon in the text. There are many other questions, which are too vague to be mentioned here.

Obviously one may ask for computations of~$\gt(G)$ or~$\gtz(G)$ for various groups~$G$, and obtaining results ``by hand'', that is, without computers, would certainly have value. For example, in a subsequent publication we shall describe~$\gtz(PSL_2(\f_q))$, confirming the pattern emerging in theorem~\ref{thm-conj-psl2q}. It would be interesting to have many other such results available, so as to understand better what it is about~$G$ that makes~$\gt(G)$ large or small. The question of {\em intuition} is wide open.

Intuition is much needed for the next problem: find an entire (infinite) family of groups~$(G_i)_{i \in I}$ such that you can compute the inverse limit 
\[ \lim_{i \in I} \, \gt(G_i) \, ;  \]
then, give an interpretation of the map~$\gal \to \gt \to \lim_i \gt(G_i)$. This would be a direct generalization of the cyclotomic character (obtained for cyclic groups, see the introduction). In this paper, in a sense, we have done this for the dihedral groups, unfortunately obtaining a trivial answer (see proposition~\ref{prop-dihedral-summary} and the subsequent discussion).

Perhaps less ambitious, but already quite hard, is the question of finding a single explicit element in~$\lim_i \gt(G_i)$. Recall that, beyond the identity and complex conjugation, exhibiting elements of~$\gal$ is no simple matter at all. With any luck, the problem at hand might be less difficult.

In \S\ref{subsec-p-groups} we have mentioned that there is no known bound on the order of the group~$\gt(G)$. It is an exciting problem to find one, since it would provide a bound for the ``order'' of the profinite group~$\gal$, in some sense.

Finally, if~$\gt$ is similar to~$\gal$ in any reasonable sense, then it should be interesting to have a look at its closed subgroups of index~$2$, or equivalently at the (continuous) homomorphisms~$\gt \to \f_2$ ; they form the group which would be written~$H^1(\gt, \f_2)$ in cohomological notation. Indeed, the cohomology ring $H^*(\gal, \f_2)$ is generated by its elements of degree~$1$ (part of Milnor's conjecture, now a theorem by Voevodsky, states that this is true for any field replacing~$\q$, showing the depth of this result). It follows from the fact that the abelianization map~$\gal \to \hat \z$ can be lifted to a map~$\gt \to \hat\z$, as we have seen, that any map~$\gal \to \f_2$ can also be lifted to~$\gt$ ; in other words, the homomorphism 
\[ H^1(\gt, \f_2) \longrightarrow H^1(\gal, \f_2)  \]
is surjective. One may ask whether it is injective as well, that is: does~$\gt$ have homomorphisms onto~$\f_2$ which are identically trivial on~$\gal$? A probably harder question being: what is the abelianization of~$\gt$?

\bibliography{myrefs}

\newcommand{\noopsort}[1]{} \newcommand{\printfirst}[2]{#1}
  \newcommand{\singleletter}[1]{#1} \newcommand{\switchargs}[2]{#2#1}
  \def\cprime{$'$}
\providecommand{\bysame}{\leavevmode\hbox to3em{\hrulefill}\thinspace}
\providecommand{\MR}{\relax\ifhmode\unskip\space\fi MR }
\providecommand{\MRhref}[2]{%
  \href{http://www.ams.org/mathscinet-getitem?mr=#1}{#2}
}
\providecommand{\href}[2]{#2}
\begin{thebibliography}{Neu99}

\bibitem[Dri90]{drinfeld}
V.~G. Drinfel{\cprime}d, \emph{On quasitriangular quasi-{H}opf algebras and on
  a group that is closely connected with {${\rm Gal}(\overline{\bf Q}/{\bf
  Q})$}}, Algebra i Analiz \textbf{2} (1990), no.~4, 149--181. \MR{1080203
  (92f:16047)}

\bibitem[Gui14]{pedro}
Pierre Guillot, \emph{An elementary approach to dessins d'enfants and the
  {G}rothendieck-{T}eichm\"uller group}, L'enseignement Math\'ematique
  \textbf{60} (2014), 293--375.

\bibitem[Isa08]{isaacs}
I.~Martin Isaacs, \emph{Finite group theory}, Graduate Studies in Mathematics,
  vol.~92, American Mathematical Society, Providence, RI, 2008. \MR{2426855
  (2009e:20029)}

\bibitem[Jon14]{gareth}
Gareth Jones, \emph{Regular dessins with a given automorphism group},
  Contemporary Mathematics \textbf{629} (2014), 245--260.

\bibitem[MZ00]{zvonkine}
Nicolas Magot and Alexander Zvonkin, \emph{Belyi functions for {A}rchimedean
  solids}, Discrete Math. \textbf{217} (2000), no.~1-3, 249--271, Formal power
  series and algebraic combinatorics (Vienna, 1997). \MR{1766270 (2001m:14043)}

\bibitem[Neu99]{neukirch}
J{\"u}rgen Neukirch, \emph{Algebraic number theory}, Grundlehren der
  Mathematischen Wissenschaften [Fundamental Principles of Mathematical
  Sciences], vol. 322, Springer-Verlag, Berlin, 1999, Translated from the 1992
  German original and with a note by Norbert Schappacher, With a foreword by G.
  Harder. \MR{1697859 (2000m:11104)}

\end{thebibliography}
\bibliographystyle{amsalpha}

\end{document}